\documentclass[11pt]{amsart}

\usepackage{hyperref}

\usepackage{xcolor}
\usepackage{enumerate}   
\usepackage{amsmath, amssymb, amsfonts, amsthm, fullpage}
\usepackage{graphics,graphicx}
\usepackage{color}
\usepackage{stmaryrd}

\usepackage{marginnote} %

\newtheorem{remark}{Remark}[section]
\newtheorem{lemma}{Lemma}[section]
\newtheorem{proposition}{Proposition}[section]
\newtheorem{theorem}{Theorem}[section]

\newcommand{\bF}{\mathbf{F}}    
\newcommand{\bu}{\mathbf{u}}
\newcommand{\Div}{\nabla\!\cdot\!}

\begin{document}
\title[Oscillatory and regularized shock waves for a dissipative Peregrine-Boussinesq system]{Oscillatory and regularized shock waves for a dissipative Peregrine-Boussinesq system}

\author{Larkspur Brudvik-Lindner}
\address{\textbf{L. Brudvik-Lindner} Victoria University of Wellington, School of Mathematics and Statistics, PO Box 600, Wellington 6140, New Zealand}
\email{brudvilark@myvuw.ac.nz}

\author{Dimitrios Mitsotakis}
\address{\textbf{D.~Mitsotakis:} Victoria University of Wellington, School of Mathematics and Statistics, PO Box 600, Wellington 6140, New Zealand}
\email{dimitrios.mitsotakis@vuw.ac.nz}

\author{Athanasios E. Tzavaras}
\address{\textbf{A.E.~Tzavaras:} King Abdullah University of Science and Technology (KAUST), Computer Electrical and Mathematical Science and Engineering Division (CEMSE), Thuwal, 23955-6900, Saudi Arabia }
\email{athanasios.tzavaras@kaust.edu.sa}

\subjclass[2000]{35Q35, 92C35, 35C07}

\date{\today}


\keywords{Boussinesq system, positive surge, diffusive-dispersive shock wave, undular bore, existence, convergence}

\begin{abstract}
We consider a dissipative, dispersive system of Boussinesq type, describing wave phenomena in settings where dissipation has an effect. Examples include undular bores in rivers or oceans where dissipation due to turbulence is important for their description. We show that the model system admits traveling wave solutions known as diffusive-dispersive shock waves, and we categorize them into oscillatory and regularized shock waves depending on the relationship between dispersion and dissipation. Comparison of numerically computed solutions with laboratory data suggests that undular bores are accurately described in a wide range of phase speeds. Undular bores are often described using the original Peregrine system which, even if not possessing traveling waves
tends to provide accurate approximations for appropriate time scales. To explain this phenomenon, we show that the error between the solutions of the dissipative versus the non-dissipative Peregrine systems are proportional to the dissipation times the observational time.
\end{abstract}

\maketitle

\section{Introduction}

Nonlinear, dispersive waves appear in settings where  nonlinear effects are relevant and waves of different wavelengths travel at different speeds. 
Examples abound in oceanic waves, atmospheric waves, including the morning glory of Australia \cite{CSR1981}, tidal bores in oceans and rivers \cite{Ch2012} and pulses in viscoelastic vessels, \cite{mitsotakis2019}.
In all these cases, traveling waves seem to play a fundamental role manifesting themselves via a wave front followed by undulations. The front is decreasing to rest, while the undulations appear as  progressively decreasing in amplitude oscillations that expand backwards for long distances. Such waves are usually called undular bores, positive surges or Favre waves \cite{Favre1935,BL1954,GP1987}. 

Undular bores have been studied using unidirectional nonlinear dispersive equations such as the Benjamin-Bona-Mahony (BBM)  equation
(also called Regularized Long Wave (RLW) equation) or the Korteweg-de-Vries (KdV) equation \cite{Pere1966,BS1985,EH2016,EHS2017}, 
as well as bidirectional nonlinear dispersive equations such as the Serre equations \cite{EGS2006, EGS2008} or  the Kaup-Boussinesq (linearly ill-posed \cite{BCS2002}) system \cite{EGK2005,K1975,GP1991}. In this work we focus on the Peregrine-Boussinesq system \cite{Pere1967,Bous1871,BCS2002} with flat bottom topography 
\begin{equation}\label{eq:peregrine}
    \begin{aligned}
        &\eta_t + D u_x + (\eta u)_x = 0\ ,\\
        &u_t + g\eta_x + uu_x -  \frac{D^2}{3}u_{xxt}  = 0\ ,
    \end{aligned}
\end{equation}
where $x$, $t$ are the space and time variables, $\eta=\eta(x,t)$ denotes the free-surface elevation while $u=u(x,t)$ stands for the horizontal velocity of the fluid measured at some height above the flat bottom of depth $D$. The constant $g$ is the acceleration due to gravity. The particular system describes the propagation of long waves in both directions, \cite{BCS2002}, which provides a more accurate description of undular bores compared to its unidirectional counterparts. 

Peregrine's system \eqref{eq:peregrine} possesses solitary wave solutions that have smooth profiles and decay exponentially in space \cite{Chen1998}. 
It is well-posed in the Hadamard sense: for initial data $(\eta(x,0),u(x,0))\in H^s\times H^{s+1}$, where $H^s=H^s(\mathbb{R})$ the usual Sobolev space with $s\geq1$,  and if $\inf_{x\in\mathbb{R}}\{1+\eta(x,0)\}>0$, then for any $T>0$ there is a unique solution $(\eta,u)\in C([0,T]; H^s\times H^{s+1})$ of (\ref{eq:peregrine}),
\cite{Sch1981,Am1984,BCS2004}, and the recent work \cite{MTZ2021} for $s>1/2$. It should be noted that the inviscid Peregrine system \eqref{eq:peregrine} does not possess undular bores as traveling wave solutions but instead possesses dispersive shock waves formed by series of solitary waves followed by dispersive tails, making the explanation of undular bores via this system (alone) questionable.

It has been experimentally confirmed that dissipation due to turbulence plays a role in the shape formation of undular bores \cite{S1965}, in addition to nonlinearity and dispersion, which are of course unavoidable components of water waves. Turbulence is expected to produce dissipation but  its effect on the shape of undular bores is known only via experimental evidence \cite{KC2008,KC2009}. The objective of this work is to test the effect of dissipation on the Peregrine system and to show that it is a necessary ingredient for the existence of traveling waves. 
We consider the system
\begin{equation}\label{eq:boussinesqs} 
    \begin{aligned}
        &\eta_t + Du_x + (\eta u)_x = 0\ ,\\
        &u_t + g\eta_x + uu_x - \frac{D^2}{3} u_{xxt} - \varepsilon u_{xx} = 0\ ,
    \end{aligned}
\end{equation}
where $\varepsilon$ is a positive constant, in order to study the effect of viscosity on (\ref{eq:peregrine}). The system \eqref{eq:boussinesqs}
can be derived to describe free-surface water waves of long wavelength and small amplitude from dissipation approximations of potential flow \cite{JW2004} using asymptotic expansions as demonstrated in \cite{DD2007}, see also Appendix \ref{appA}.
Similar Boussinesq systems are derived in other settings for laminar flows with parabolic velocity profile, \cite{DD2007,mitsotakis2019}, 
see also \cite{LO2004,HL2015,DG2016}.  The selection of the parameter $\varepsilon$ relies on experimentation because the dissipative effects are known only via experimental evidence and are not well-formulated in mathematical terms \cite{KC2008,KC2009}. Note that the term $u_{xx}$ describes the bulk damping in fluid's body. In more realistic situations, dissipation due to boundary layers and bottom friction must be included via the heuristic Manning's frictional term $c_mu|u|/(D+\eta)^{1/3}$, \cite{dkm2011}. Inclusion of such term will result eventually in the dissipation of any solution in time as it is demonstrated experimentally in \cite{BPS1981}.

We  study traveling wave solutions \eqref{eq:boussinesqs}  and their dependence as functions
of the parameters $\varepsilon$ and $\delta=\tfrac{D^2}{3}$. 
Such traveling waves often approximate shocks and they are called diffusive-dispersive shock waves \cite{EHS2017}. 
We show their shape depends on the balance among dispersion and dissipation:  (i) When dispersion
is weak, then the wave is monotone. (ii)  When the dispersion is moderate 
then the traveling waves consist of undulations following the traveling front.

Our results are motivated and the techniques parallel the analysis of \cite{BS1985} where
existence and the shape of traveling waves is studied for  a  KdV-Burgers equation. But there is a difference; as the first equation has no dissipation
it leads to a singularity in the differential equations describing the traveling waves, that needs to be analyzed
and causes the free surface elevation to have steep gradients for large speeds;
this phenomenon is not present in the KdV-Burgers equation of \cite{BS1985}.
It is convenient to consider a scaled variant of (\ref{eq:peregrine}), with $g=1$ and $\delta=D^2/3$, where the speed $c$ 
is scaled by $\sqrt{gD}$ and leads to the Froude numbe.  In the long wave limit we expect that traveling waves propagate with scaled speed $c>1$.
We show that the system (\ref{eq:boussinesqs}) has a unique up to horizontal translations traveling wave solution that resembles positive surge for each (scaled by $\sqrt{gD}$) phase speed $c>1$. These traveling waves are oscillatory shock waves when $\varepsilon^2<4\delta c\alpha(c)$ and regularized shock waves when $\varepsilon^2\geq 4\delta c\alpha(c)$ where
$$\alpha(c)=\frac{c - \sqrt{ c^2+8}}{2}+\frac{4 c}{(c - \sqrt{ c^2+8})^2} \ .$$ The scaled speed $c$ is referred in the literature as the Froude number.  
There are two asymptotic regimes as $\delta,\varepsilon \to 0$.  The weak dispersion regime $\delta < \frac{\varepsilon^2}{4 c \alpha(c)}$, characterized by the
eigenvalues of a linearized operator being real, where the traveling wave converges monotonically to a shock of the associated shallow water system. 
Moreover, combined with the analysis in \cite{BMT2023}, there is a moderate dispersion regime $ \frac{\varepsilon^2}{4 c \alpha(c)} < \delta \le o(\varepsilon)$, 
where  the diffusive-dispersive traveling wave still converges to a classical isentropic shock
but the tail exhibits an oscillatory wave-train that shrinks as $\delta,\varepsilon \to 0$.

From a physical point of view, fitting solutions of the dissipative Boussinesq system with laboratory data suggest the use of certain values for the parameter 
$\varepsilon$ depending on the Froude number $c$. In particular, when $c<1.3$ very little dissipation ($\varepsilon=0.05$) is adequate for the accurate description of undular bores. For larger values of $c$ when dissipation due to wave breaking becomes dominant, then values $\varepsilon=O(1)$ are necessary for a good approximation of breaking waves even for $c=1.45$, while for $c=1.6$ the current theory fails to predict accurately the fully-breaking bore.

Existence of smooth solutions has been proved for dissipative Boussinesq systems, \cite{AAG2007}.
In a final section we compare smooth solutions of the two systems (\ref{eq:peregrine}) and (\ref{eq:boussinesqs}) and show
that solutions emanating  from the same initial data  remain close with error $O(\varepsilon t)$. We note that the Peregrine system does not possess
traveling wave solutions of undular bore type, yet it is commonly used in practice for the description of undular bores.
The error estimate explains the accuracy of the asymptotic models in describing dispersive shock waves for time scales of $O(1/\varepsilon)$. Due to the fact that the Peregrine system along with all the other Boussinesq systems of \cite{BCS2002} are asymptotically equivalent with the same error estimates against the Euler equations, we expect that all these system have the same capabilities to approximate undular bores.

The structure of this paper is as follows: In Section \ref{sec:existence} we prove the existence and uniqueness of the traveling wave solutions for the system (\ref{eq:boussinesqs}). We also provide estimates for the amplitude, tail and phase speed of the waves. In Section \ref{sec:shape} we describe two different types of traveling waves and we give the exact criterion for acquiring regularized shock waves to the non-dispersive shallow water equations by dispersive-dissipative regularization. Section \ref{sec:errors}  provides the error estimate between systems (\ref{eq:peregrine}) and (\ref{eq:boussinesqs}). The article closes with conclusions and an Appendix containing a brief review of the derivation of the dissipative Boussinesq system (\ref{eq:boussinesqs}).


\section{Existence of traveling wave solutions}\label{sec:existence}

In this section we study the existence and uniqueness  (up to horizontal translations) of traveling wave solutions to the following dissipative Peregrine-Boussinesq system
\begin{equation}\label{eq:boussinesqs2} 
    \begin{aligned}
        &\eta_t + u_x + (\eta u)_x = 0\ ,\\
        &u_t + \eta_x + uu_x - \delta u_{xxt} - \varepsilon u_{xx} = 0\ ,
    \end{aligned}
\end{equation}
We will be considering traveling wave solutions to the Boussinesq system (\ref{eq:boussinesqs2}) of the form
\begin{equation}\label{eq:ansatz}
    \eta(x,t) = \eta(\xi),\quad u(x,t) = u(\xi)\ ,
\end{equation}
where $\xi=x-ct$, for $c>1$. 

The following limiting conditions are imposed on $\eta$ and $u$:
\begin{equation}\label{eq:condition1}
\lim_{\xi \to -\infty} (\eta(\xi),u(\xi)) = (\eta^-,u^-), \quad 
\lim_{\xi \to +\infty} (\eta(\xi),u(\xi)) = (\eta^+,u^+) = (0,0)\ ,
\end{equation}
where $\eta^->0$, $u^->0$, and we also require
\begin{equation}\label{eq:condition2}
    \lim_{\xi \to \pm\infty} (\eta^{(j)}(\xi),u^{(j)}(\xi)) = 0,\quad \text{for}\quad j = 1,2,\dots\ .
\end{equation}

Substitution of (\ref{eq:ansatz}) into the system (\ref{eq:boussinesqs2}), and integrating over $[\xi, \infty)$ we obtain the algebraic relation between $\eta$ and $u$
\begin{equation}\label{eq:relation1}
\eta = \frac{u}{c-u}\ , 
\end{equation}
and the ordinary differential equation
\begin{equation}\label{eq:difeq}
-cu + \eta + \frac{1}{2}u^2 + \delta cu'' - \varepsilon u' = 0\ ,
\end{equation}
where $'$ stands for the usual derivative $d/d\xi$.
Substitution of (\ref{eq:relation1}) into (\ref{eq:difeq}) gives the second order equation
\begin{equation}\label{eq:main}
    -cu - \frac{u}{u-c} + \frac{1}{2}u^2 + \delta cu'' - \varepsilon u' = 0\ .
\end{equation}
\begin{remark}\label{rem:hypothesis1}
Note that for $\eta(\xi)>0$ and $u(\xi)>0$ we have from (\ref{eq:relation1}) that $u(\xi)-c<0$ for all $x\in \mathbb{R}$. Thus, $\lim_{\xi \to-\infty}(u(\xi)- c)<0$ in order for $\lim_{\xi\to -\infty}\eta(\xi)<\infty$.
\end{remark}
\begin{remark}
We exclude the case where $\eta<0$ because we look for solutions that satisfy the asymptotic conditions (\ref{eq:condition1}) with $\eta^->0$. On the other hand, if it was $\eta>0$ and $u<0$, then from (\ref{eq:relation1}) we would have $u-c>0$, which for $c>1$ means that $u>1$, which contradicts with the asymptotic condition (\ref{eq:condition1}). The only acceptable case is when $\eta>0$ and $u>0$.  
\end{remark}
\begin{remark}\label{rem:hypothesis2}
If we can establish the existence of positive solutions (that is, $\eta> 0$, $u>0$, $c>1$ that satisfy  the conditions (\ref{eq:condition1}) and (\ref{eq:condition2})) then by symmetry we have that $\Tilde{u}(-\xi) = -u(\xi) <0$, $\Tilde{\eta}(-\xi) = \eta(\xi) > 0$ is also a solution to (\ref{eq:main}). This is a traveling wave that travels to the left with phase speed $\Tilde{c} = -c < -1$. For this reason, we search only for solutions with $u>0$ and $c>1$.
\end{remark}
Assuming that the equation (\ref{eq:main}) has positive solutions that satisfy the conditions (\ref{eq:condition1}) and (\ref{eq:condition2}), we estimate the magnitude of the tail of the traveling wave (undular bore). In particular, we have the following lemma:
\begin{lemma}\label{lem:asymptot1}
Let $u$ be a positive solution to (\ref{eq:main}) for any $c>1$, then  $$\lim_{\xi\to-\infty}u(\xi) = u_0:= \frac{3c -\sqrt{c^2 + 8}}{2}>0\ .$$
\end{lemma}

\begin{proof}
In equation (\ref{eq:main}), let $\xi\to-\infty$ to obtain
\begin{equation}
    -cu_0 - \frac{u_0}{u_0 - c} + \frac{1}{2}u_0^2 = 0\ .
\end{equation}
Then either $u_0 = 0$   or
\begin{equation}\label{eq:quadratic}
    u_0^2 - 3cu_0 + 2(c^2 - 1) = 0\ ,
\end{equation}
which has roots 
\begin{equation}\label{eq:critpo}
 u_{\pm}= \frac{3c \pm \sqrt{c^2+8}}{2}\ .
\end{equation}
Note that $c-u_0$ must be positive for positive solutions due to Remark \ref{rem:hypothesis1}. It follows from (10) that this is true (for $c>1$) only if
\begin{equation}\label{eq:leftlimit}
u_0=\frac{3c-\sqrt{c^2+8}}{2}>0\ ,
\end{equation}
excluding the possibility $u_0=u_{+}$. 

Multiplying (\ref{eq:main}) by $u'$ we obtain the identity
\begin{equation}
\left ( -\frac{c u^2}{2} - u + c \ln \frac{c}{c-u} + \frac{1}{6} u^3 + \frac{\delta c}{2} (u^\prime)^2 \right )^\prime = \varepsilon (u^\prime)^2 
\end{equation}
Integrating over the interval ($-\infty, \infty$) yields
\begin{equation}\label{eq:relation2}
    \varepsilon \int_{-\infty}^{\infty} (u')^2~d\xi = u_0+\frac{1}{2}cu_0^2- \frac{1}{6}u_0^3 - c\ln\left(\frac{c}{c-u_0}\right) \ .
\end{equation}
Note that we do not use absolute value in the logarithm since by Remark \ref{rem:hypothesis1} we have that $c-u_0>0$.
On the other hand, if $u_0 = 0$ then (\ref{eq:relation2}) implies that $u'(\xi) = 0$ for all $\xi\in\mathbb{R}$, and thus $u$ is constant, contradicting the hypothesis. Obliged to exclude the case $u_0=0$,
we substitute (\ref{eq:leftlimit}) into (\ref{eq:relation2}) to obtain for all $c>1$
$$
 \varepsilon\int_{-\infty}^{\infty}(u')^2~d\xi  =-\frac{1}{6} (2 + c^2) (-3 c + \sqrt{c^2+8}) - 
 c \ln\left[\frac{1}{4} c (c + \sqrt{c^2+8})\right]>0\ .
$$
To prove that the right-hand side is positive, set it equal to $f(c)$ and compute $f'$ and $f''$, and see that $f(1) = f'(1)=0$ and $f''(c)>0$ for $c>1$.
\end{proof}

\begin{remark}\label{rem:boundcrit}
We note that $c-1<u_0<c$ for all $c> 1$ and denote
$$\eta_0=\lim_{x\to+\infty}\eta(\xi)=\frac{u_0}{u_0-c}>0\ .$$
The speed of propagation $c$ and  the limit value $u_0$ are independent of $\varepsilon$ and $\delta$ and coincide with those of the corresponding classical shock waves for the nonlinear shallow water equations.
\end{remark}

We first estimate the local extrema of the solution $u$.

\begin{lemma}\label{lem:critical} Let $u$ be a positive, non-constant solution to (\ref{eq:main}) and suppose that $u'(x_0) = 0$ for some $x_0 \in \mathbb{R}$. Then $x_0$ is an isolated extremum. Moreover, $u(x_0)$ is a local minimum if $u(x_0)<u_0$, while $u(x_0)$ is a local maximum if $u(x_0)>u_0$.
\end{lemma}

\begin{proof}
If $u'(x_0) = 0$ then from equation (\ref{eq:main}) we have
\begin{equation}\label{eq:extrema}
    \delta cu''(x_0) = cu(x_0) - \frac{u(x_0)}{c - u(x_0)} - \frac{1}{2}u^2(x_0)\ .
\end{equation}

Local maxima of the solution occur at points $x_0$ such that $u''(x_0)<0$. This means that $u(x_0)$ satisfies the quadratic inequality 
$$u^2(x_0)-3cu(x_0)+2(c^2-1)<0\ .$$ 
Thus, we have $u(x_0)>u_0$. Similarly, the minimum values should satisfy $u''(x_0)>0$ which yields $0<u(x_0) <u_0$.

Lastly, note that if $u''(\xi_0) = 0$, then $u$ is a constant function, contradicting the hypothesis. Thus $u''(x_0) \neq 0$ and $x_0$ is an isolated extremal point.
\end{proof}

We define a dependent variable $v = (\delta c)u'$. Then (\ref{eq:main}) is equivalent to the first order system
\begin{equation}\label{eq:firstordera}
 \begin{aligned}
        &u' = (\delta c)^{-1}v\ , \\
        &v' = cu + \frac{u}{u-c} - \frac{1}{2}u^2 + \varepsilon(\delta c)^{-1}v\ ,
    \end{aligned}
\end{equation}
which we will write in the compact form
\begin{equation}\label{eq:firstorder}
\bu'=\bF(\bu)\ ,
\end{equation}
where
$$
\bu=\begin{pmatrix}
u \\ v
\end{pmatrix},\quad\text{and}\quad \bF(\bu)=\begin{pmatrix}
(\delta c)^{-1}v \\
 cu + \frac{u}{u-c} - \frac{1}{2}u^2 + \varepsilon(\delta c)^{-1}v
\end{pmatrix}\ .
$$
The system (\ref{eq:firstorder}) has three critical points: (0,0) and ($u_{\pm}$, 0). 

When linearized about (0,0), system (\ref{eq:firstorder}) has characteristic polynomial
$$\lambda^2-\varepsilon(\delta c)^{-1}\lambda-(\delta c)^{-1}\frac{c^2-1}{c}=0\ , $$ with eigenvalues
\begin{equation}\label{eq:eigen1}
    \lambda_{\pm} = \frac{\varepsilon \pm \sqrt{\varepsilon^2 + 4\delta(c^2 - 1)}}{2\delta c}\ .
\end{equation}
Hence (0,0) is a saddle point. When the system is linearized about $\bu_{-}=(u_{-},0)$, the Jacobian matrix is $$\bF'(\bu_{-})=\begin{pmatrix}
0 & (\delta c)^{-1}\\
-\alpha(c) & \varepsilon (\delta c)^{-1}
\end{pmatrix}\ ,$$
where (recalling that $u_0=u_{-}$)
$$\alpha(c)=\frac{c - \sqrt{ c^2+8}}{2}+\frac{4 c}{(c - \sqrt{ c^2+8})^2}=u_0-c+\frac{c}{(u_0-c)^2}\ .$$
The characteristic polynomial of the Jacobian matrix is $\lambda^2-\varepsilon(\delta c)^{-1}\lambda +\alpha(c)~(\delta c)^{-1}=0$ with eigenvalues
\begin{equation}\label{eq:eigen2}
    \Lambda_{\pm}=\frac{\varepsilon\pm\sqrt{\varepsilon^2-4\delta c\alpha(c)}}{2\delta c}\ .
\end{equation}
 These eigenvalues are complex when
\begin{equation}\label{eq:relation3}
    \varepsilon^2<4\delta c\alpha(c)\ ,
\end{equation}
in which case the critical point $(u_0,0)$ is a spiral. The first derivative of the function
\begin{equation}
   \alpha(c)= \frac{c - \sqrt{ c^2+8}}{2}+\frac{4 c}{(c - \sqrt{ c^2+8})^2}\ ,
\end{equation}
is greater than zero for $c > 1$. Since $\alpha(c)$ is equal to zero at $c = 1$, we conclude that the function is increasing for $c \geq 1$, and hence the right-hand side of (\ref{eq:relation3}) is positive for $c>1$.  The real part of the eigenvalues $\Lambda_{\pm}$ is  positive and the critical point is always unstable (node or spiral).

The third critical point $u_{+} > 0$ satisfies $u_+ > c$. As we see below, solutions to (\ref{eq:firstorder}) must satisfy $u(\xi) < c$ for all $\xi \in \mathbb{R}$. Thus the critical point $u_{+}$ cannot be connected to $(0,0)$ at $+\infty$ and will be excluded from the possible asymptotic limits.

Now let $\mathcal{R}=\{(u(\xi),v(\xi)):\xi\in\mathbb{R}\}$ be any bounded orbit of the system (\ref{eq:firstorder}). The following will be a study of the asymptotic states of $\mathcal{R}$ at $+\infty$ and $-\infty$. However first note that (\ref{eq:firstorder}) has no non-constant periodic solutions. To see this, let ($u_p, v_p$) be a periodic solution to (\ref{eq:firstorder}) with period $p>0$. Then ($u_p, v_p$) is also a solution to (\ref{eq:main}). Multiplying (\ref{eq:main}) by $u'_p$ and integrating over the interval (0, $p$) yields the identity
\begin{equation}\label{eq:period}
    \varepsilon\int_{0}^{p}(u'_p)^2~d\xi  = 0\ ,
\end{equation}
since by periodicity, all terms integrate to zero except one. It follows from (\ref{eq:period}) that $u_p$ is constant. Alternatively, observe that $\Div \bF(\bu)=\varepsilon(\delta c)^{-1}>0$ for all $c>1$, and thus by Bendixson-Dulac Theorem, \cite{W2003}, system (\ref{eq:firstorder}) has no closed orbits in its phase space. This excludes also the existence of classical solitary wave solutions, which are homoclinic orbits to the origin. 

\begin{remark}\label{rem:conditions}
Let $(u, v)$ be a solution to (\ref{eq:firstorder}). Both $u(\xi)$ and $v(\xi)$ tend to 0 as $\xi \to +\infty$ so it follows from (\ref{eq:firstorder}) that $u'(\xi)$ and $v'(\xi)$ also tend to 0 as $\xi \to +\infty$. By induction, $u^{(j)}(\xi) \to 0$ and $v^{(j+1)}(\xi) = (\delta c)u^{(j)}(\xi)\to 0$ as $\xi \to +\infty$ for all $j \geq 0$.
\end{remark}

\begin{lemma}\label{lem:lem2.3} If $(u,v)$ is a solution of the system (\ref{eq:firstorder}) for $c>1$ that satisfies the asymptotic conditions (\ref{eq:condition1}), then $u$ is a positive solution in the sense that $u(\xi)>0$ and $u(\xi)<c$ for all $\xi\in \mathbb{R}$.
\end{lemma}

\begin{proof}
Let $(u, v)$ be a solution to (\ref{eq:firstorder}). Then $u(\xi)$  satisfies (\ref{eq:main}). Multiplying (\ref{eq:main}) by $u'$ and integrating over the interval [$y, \infty$) we obtain
\begin{equation}\label{eq:fantasticineq}
    \frac{c}{2}u^2(y) - \frac{1}{6}u^3(y)+u(y)-c\ln\frac{c}{|u(y)-c|} = \varepsilon\int_{y}^{\infty}(u')^2~d\xi + \frac{\delta c}{2}[u'(y)]^2>0\ .
\end{equation}
If $u(\xi)$ is not constant then Lemma \ref{lem:critical} implies that the right hand side of (\ref{eq:fantasticineq}) is positive. Since $\lim_{\xi\to+\infty}u(\xi)=0$, we conclude that $u(\xi)< c$ for all $\xi\in \mathbb{R}$. To see the last assertion, assume that there is a $y_0\in \mathbb{R}$ such that $u(y_0)>c$ and a $y_1\in\mathbb{R}$ such that $u(y_1)<c$ then by continuity of the solution $u$ there will be a $y_2\in(y_0,y_1)$ such that $u(y_2)=c$. Taking the limit $y \to y_2$ in (\ref{eq:fantasticineq}), the left-hand side tends to $-\infty$ which is a contradiction. Thus, any smooth solution $u$ must satisfy $u(y)<c$ for all $y\in\mathbb{R}$.
For such a solution, using the inequality
\begin{equation}\label{eq:logineq}
    \frac{x}{1+x}\leq \ln(1+x)\ , \quad \mbox{$x>0$},
\end{equation}
upon setting $1+x=\frac{c}{|u-c|}$ 
 (\ref{eq:logineq}) we obtain
\begin{equation}\label{eq:inequality2}
    \frac{u^2}{6}(3c-u)=\frac{c}{2}u^2(y) - \frac{1}{6}u^3(y)>0\ .
\end{equation}
We conclude that $u(\xi)$ is either positive or negative. (Inequality (\ref{eq:inequality2}) also yields that $3c-u(y)>0$). However by Lemma 2.1, the solution $u$ cannot be negative because the limit $\lim_{\xi\to-\infty}u(\xi)=u_0>0$. If $u$ were negative it would need to cross the $x$-axis in order to tend to a positive number as $\xi\to-\infty$ thus violating 
\eqref{eq:inequality2}. Therefore, the solution $u(\xi)$ must be positive. 
\end{proof}

Lemma \ref{lem:lem2.3} shows that the set $\mathcal{M}=\{(u,v)\in\mathbb{R}^2: 0<u<c \}$ is an invariant set. Since there are no non-trivial periodic solutions to (\ref{eq:firstorder}) and we have that at infinity both $u'$ and $v'$ tend to zero, we conclude that both the $\alpha$-limit set and the $\omega$-limit set must contain critical points of the system. Since these limit sets are connected by the orbit $\mathcal{R}$, they must each contain exactly one critical point and hence the orbit must tend asymptotically to a critical point both at $+\infty$ and $-\infty$. At one side the orbit connects to $(0,0)$ due to  (\ref{eq:condition2}) while at $-\infty$ the critical point must be different due to Lemma \ref{lem:asymptot1}. Using a refined version Poincaré-Bendixson Theorem, \cite{W2003}, any bounded orbit of (\ref{eq:firstorder}) connects the critical points $(0,0)$ and $(u_0,0)$. 

Changing the direction of the traveling wave using the variables $\tilde{\xi}= -\frac{\xi}{\delta c}$, $\tilde{v}= v$ and $\tilde{u}= u$, we  write the system (\ref{eq:firstordera}) as
\begin{equation}\label{eq:systemb}
\begin{aligned}
\tilde{u}' &= -\tilde{v}\ ,\\
\tilde{v}' &= \frac{dG}{du} (\tilde{u})  - \varepsilon \tilde{v}\ .
\end{aligned}
\end{equation}
where ${}^\prime = \frac{d}{d \tilde{\xi}}$ while $G(u)$ is the potential function
\begin{equation}\label{eq:potentialf}
G(u)=(\delta c) \left[\frac{1}{6}u^3-\frac{c}{2}u^2-u+c\ln\frac{c}{|c-u|} \right] \ .
\end{equation}
Observe that (\ref{eq:systemb}) is dissipative in the sense that its solutions $(\tilde{u}(\tilde{\xi}),\tilde{v}(\tilde{\xi}))$ satisfy $$\frac{d}{d\tilde{\xi}} V(\tilde{u},\tilde{v})=-\varepsilon \tilde{v}^2\leq 0\quad \text{ for } \quad \tilde{\xi} \in \mathbb{R}\ ,$$ 
where $V$ is the Liapunov function
\begin{equation}\label{eq:liapunov}
V(\tilde{u},\tilde{v})=\frac{\tilde{v}^2}{2}+G(\tilde{u})\ .
\end{equation}
The potential is given within a normalization constant which is here selected so that $G$ is negative for the traveling waves under study.

By the LaSalle invariance principle, \cite{W2003}, every bounded orbit $\mathcal{R}$ of (\ref{eq:firstordera}) converges to the equilibrium point $(u_0,0)$ as $\xi\to-\infty$. Moreover, when $\varepsilon=0$ the systems (\ref{eq:firstordera}) and (\ref{eq:systemb}) are Hamiltonian and $V$ coincides with their Hamiltonian. In such case, we know that for $c>1$ there is a unique (up to translations) traveling wave (homoclinic orbit to the origin) with infinite period and of maximal energy $G(\bar{u})=0$, where here $\bar{u}=\max_\xi(u)$ denotes the amplitude of the traveling wave \cite{Chen1998,DM2008}. The particular equation $G(\bar{u})=0$ is known as the speed-amplitude relation for solitary waves ($\varepsilon=0$). Let us denote by $\bar{\eta}=\max_\xi(\eta)$ the amplitude of $\eta$.  It satisfies  the formula $\bar{\eta}=\bar{u}/(c-\bar{u})$. Solving the equation $G(\bar{u})=0$ we obtain the speed-amplidute relationship for the solitary wave
\begin{equation}\label{eq:spamp}
c=\frac{\sqrt{6}(1+\bar{\eta})}{\sqrt{3+2\bar{\eta}}}\cdot \frac{\sqrt{(1+\bar{\eta})\ln (1+\bar{\eta})-\bar{\eta}}}{\bar{\eta}}= 1+\frac{1}{2}\bar{\eta}-\frac{5}{24}\bar{\eta}^2+\frac{79}{720}\bar{\eta}^3+O(\bar{\eta}^4)\ .
\end{equation}
Note that this relationship does not depend neither on $\delta$ nor on $\varepsilon$. In the case of diffusive-dispersive shock waves, their amplitudes should vary between $\eta_0=u_0/(c-u_0)$ and the value $\bar{\eta}$ of the solitary wave amplitude. As we shall see later, even in the case where the dissipation of $O(10^{-2})$, the experimental data confirm that the amplitude of the solitary wave is a very good approximation of the amplitude of undular bores. 


$G(u)$ has an inflection point at $(u_c,G(u_c))$ with $u_c=c-\sqrt[3]{c}$. It also has three extrema $(0,0)$ and $(u_{\pm},G(u_{\pm}))$, though we have excluded the case of $u_+$. Moreover, $G(u)$ has a singularity (vertical asymptote) at $u=c$ with $\lim_{u\to c}G(u)= \infty$. A sketch of the potential is presented in Figure \ref{fig:potential1}. The traveling wave will take values in the region between the local maximum of $G$ and the upper bound of possible maximum amplitude of a traveling wave, depicted in the figure by $\bar{u}$. The local maxima of the traveling wave will be located on the right side of the minimum $u_0$ of $G$, while the local minima of $u$ on its left side. As we shall see in Proposition \ref{prop:maxampl} the maximum amplitude of the traveling wave cannot exceed the point of intersection between the line $y=0$ and the potential function $G$. The line $y=0$ depicts also the barrier of the maximal energy a solution may have.

\begin{figure}[ht!]
\includegraphics[width=\textwidth]{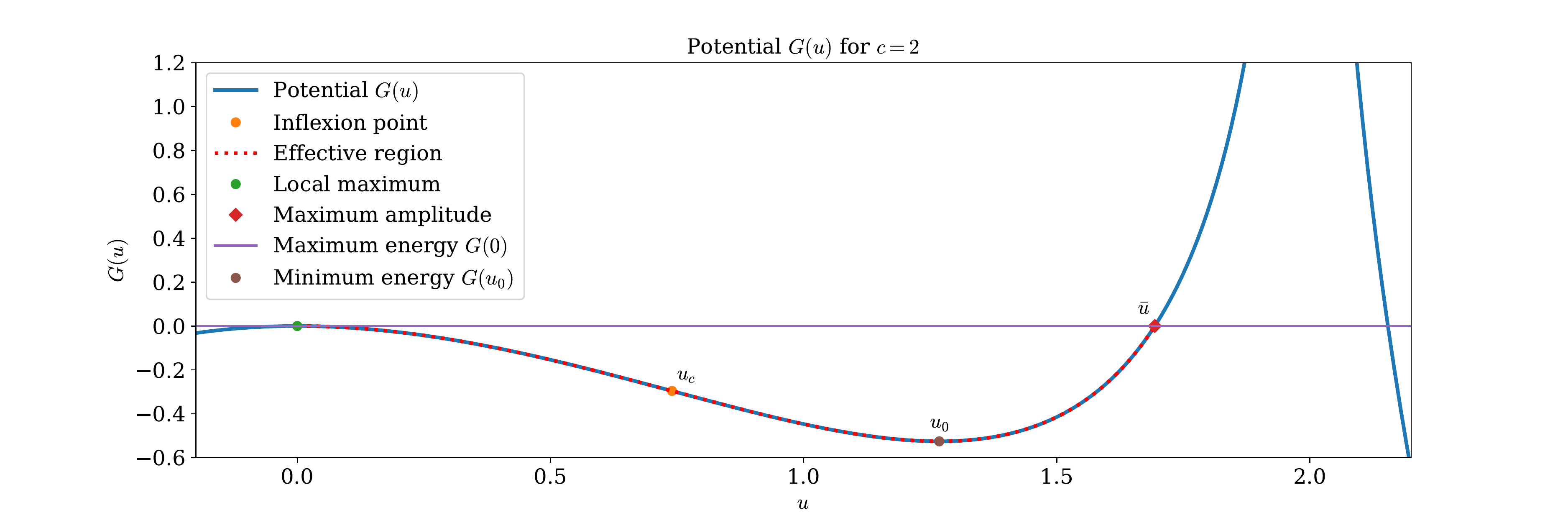}
\caption{The potential function $G(u)$ for $c=2$ and $\delta=1/2$. If $u$ is a traveling wave then $G(u(\xi))$ remains bellow the line $y=0$ }\label{fig:potential1}
\end{figure}

A consequence of the decreasing Liapunov function is the following Proposition on positive solutions.
\begin{proposition}\label{prop:maxampl}
There is $\bar{u}\in (u_0,c)$ such that $0<u(\xi)<\bar{u}<c$ for all $\xi\in\mathbb{R}$. Specifically, for all $\xi\in \mathbb{R}$ it holds that $G(u(\xi))<0$. In addition, if $u(x_0)$ is a local maximum of $u(\xi)$ then $u(x_0)\sim c$ as $c\to\infty$.
\end{proposition}
\begin{proof}
Let $\tilde{\xi}\in \mathbb{R}$, then $|\tilde{u}(\tilde{\xi})|+|\tilde{v}(\tilde{\xi})|\to 0$ as $\tilde{\xi}\to -\infty$ (since we have changed its direction). Because $V(\tilde{\xi})$ is decreasing, then for $\tilde{\xi}>-\infty$ we have  $V(\tilde{u}(\tilde{\xi}),0)<V(0,0)$.  This yields,
$$G(\tilde{u}(\tilde{\xi}))< G(0)=0,\quad \text{ for all }\ \xi\in\mathbb{R}\ .$$
Thus, if $\tilde{u}(\tilde{x}_0)$ is any local maximum, this must satisfy $G(\tilde{u}(\tilde{x}_0))<0$. If we denote $\bar{u}>0$ the value for which $G(\bar{u})=0$, then for $x_0=-(\delta c) \tilde{x}_0$ we have 
\begin{equation}\label{eq:maximampl}
u_0<u(x_0)<\bar{u}<c\ .
\end{equation}
Furthermore, since 
$$\lim_{c\to\infty}\frac{u_0}{c}=\lim_{c\to\infty}\frac{3c-\sqrt{c^2+8}}{2c}=1\ ,$$
then by (\ref{eq:maximampl}) we have
$$\frac{u_0}{c}<\frac{u(x_0)}{c}<\frac{\bar{u}}{c}<1\ .$$
Thus, for large values of $c$, $u(x_0)\sim c$, and at the same time $G(u(x_0))<0$. This means that the solution will be concentrated close to $\bar{u}\sim c$ for large values of $c$, and so fast that the solution tends to be weakly singular in the limit $c\to\infty$. The difference between the profiles of traveling waves of different speeds are presented in Figure \ref{fig:dsw1}\end{proof}

Because $(0,0)$ is a saddle point, there are two semi-orbits that converge to $(0,0)$ as $\xi\to+\infty$ and approach the origin tangentially to the stable manifold determined by the slope
$$\lim_{\xi\to+\infty}\frac{v'(\xi)}{u'(\xi)}=\lim_{\xi\to+\infty}\frac{v(\xi)}{u(\xi)}=\delta c \lambda_{-}\ ,$$
where $\lambda_{\pm}$ as defined in (\ref{eq:eigen1}). Thus, one semi-orbit approaches the origin from the fourth quadrant $Q_4=\{(u,v):u>0, v<0\}$ and the other through the second quadrant $Q_2=\{(u,v):u<0, v>0\}$, \cite{W2003}. The later case of negative $u$ is excluded though due to Lemma \ref{lem:lem2.3} and thus the only semi-orbit that can exist is the one with $u>0$.

We have seen that if equation (\ref{eq:main}), or the equivalent system (\ref{eq:firstorder}), has a solution $u$ satisfying the asymptotic conditions (\ref{eq:condition1}), then $0 < u < c$ ensuring that $\eta > 0$, see Lemma \ref{lem:lem2.3}  and \eqref{eq:relation1}.
We next prove that the system (\ref{eq:firstorder}) has at least one such solution $u$. The main result of this section is stated bellow.

\begin{theorem}\label{thrm:existence}
Let $\delta>0$, $\varepsilon>0$ and $c>1$ be given constants. Then there is a unique, global solution $(u,v)$ of the system (\ref{eq:firstorder}), which defines in the phase space a bounded orbit $\mathcal{R}\subset \mathcal{M}$ that tends to $(u_0,0)$ as $\xi\to-\infty$, and to $(0,0)$ as $\xi\to+\infty$. 
\end{theorem}
\begin{proof}
Let $(u(\xi),v(\xi))$ be a solution of (\ref{eq:firstorder}) corresponding to one of the two semi-orbits lying on the stable manifold of the saddle $(0,0)$, and assume that it is defined for at least large values of $\xi$ for the particular values of $c$. Such a solution exists since the function $\bF(\bu)$ of (\ref{eq:firstorder}) is locally Lipschitz. Specifically, for $v\in \mathbb{R}$ and $u\in I\subset (0,c)$ we have that $\bF$ is continuously differentiable with
$$
\big(\nabla \bF (\bu^\ast) \big)\bu=\begin{pmatrix}(\delta c)^{-1}v, & cu -\frac{c  u}{(u^\ast -c)^2} - u^\ast u+\varepsilon (\delta c)^{-1} v
\end{pmatrix}^T\ . $$
Thus, from the standard theory of differential equations we deduce that there is a unique (subject to $\xi$-translations) solution to (\ref{eq:firstorder}) for as long as it is remained bounded. 
Therefore, the semi-orbit that approaches the origin exists for at least large values of $\xi$. By the Remark \ref{rem:conditions} such orbit satisfies the boundary condition (\ref{eq:condition2}). By Lemma \ref{lem:lem2.3} we also have that $u$ will be always positive and bounded by $c$. In addition to $u$, if $v$ can be shown to be bounded for all values of $\xi$, then the global existence of the traveling waves will be established.

We first show that 
\begin{equation}\label{eq:loboundv}
v(\xi)\geq -\frac{\delta c}{\varepsilon}(2-3\sqrt[3]{c^2}+c^2)\geq -\frac{\delta c^3}{\varepsilon}\ .
\end{equation}
For contradiction, we assume that for some $\xi\in\mathbb{R}$ it holds
$v(\xi)<-(\delta c)(2-3\sqrt[3]{c^2}+c^2)/\varepsilon$. Let $\xi_0$ be the maximal $\xi$ such that $v(\xi_0)=-\mu$ with $-\mu<-(\delta c)(2-3\sqrt[3]{c^2}+c^2)/\varepsilon<0$. Because $v(\xi)\to 0>-\mu$ as $\xi\to+\infty$, we have that $v(\xi)>v(\xi_0)$ for $\xi>\xi_0$, and specifically $v'(\xi_0)>0$ since $\xi_0$ was assumed to be maximal. On the other hand, from the second equation of (\ref{eq:firstordera}) and using $0<u(\xi_0)<c$, we have that
$$
\begin{aligned}
v'(\xi_0) &= cu(\xi_0)+\frac{u(\xi_0)}{u(\xi_0)-c}-\frac{1}{2}u^2(\xi_0)+\varepsilon(\delta c)^{-1} v(\xi_0)\\
&=cu(\xi_0)+\frac{u(\xi_0)}{u(\xi_0)-c}-\frac{1}{2}u^2(\xi_0)-\varepsilon(\delta c)^{-1} \mu \\
&=f(u(\xi_0))-\varepsilon(\delta c)^{-1} \mu 
\end{aligned}
$$
where $f(u)=cu+\frac{u}{u-c}-\frac{1}{2}u^2$. Note that $f(u)=-G'(u)/(\delta c)$. It is easilly seen that $f'(u)=-\frac{(u-c)^3+c}{(u-c)^2}$, and thus $f(u)$ has a global maximum at $u_c=c-\sqrt[3]{c}$ with maximum value 
$$f(u_c)=(2-3\sqrt[3]{c^2}+c^2)/2>0\ .$$ Thus, for $c>1$ we have
$$\begin{aligned}
v'(\xi_0)&\leq f(u_c)-\varepsilon(\delta c)^{-1} \mu\\
&<(2-3\sqrt[3]{c^2}+c^2)/2-(2-3\sqrt[3]{c^2}+c^2)\\
&=-(2-3\sqrt[3]{c^2}+c^2)/2\\
&<0
\end{aligned}
$$
which is a contradiction.

Similarly, we show that 
$$v(\xi)<C:=\frac{c/(c-\bar{u})+c^2/2}{\varepsilon(\delta c)^{-1}}\ ,$$
where $0<\bar{u}<c$ is the actual upper bound of $u$ with  $u\leq \bar{u}<c$ due to the Proposition \ref{prop:maxampl}. To see this assume that $\xi_1$ is the largest value of $\xi$ such that $v(\xi)=C>0$. Since $v(\xi)\to 0$ as $\xi\to+\infty$ yields that $v'(\xi_1)<0$. But from the second equation of (\ref{eq:firstordera}) we have that
$$
\begin{aligned}
v'(\xi_1) &= cu(\xi_1)+\frac{u(\xi_1)}{u(\xi_1)-c}-\frac{1}{2}u^2(\xi_1)+\varepsilon(\delta c)^{-1} v(\xi_1)\\
&=cu(\xi_1)+\frac{u(\xi_1)}{u(\xi_1)-c}-\frac{1}{2}u^2(\xi_1)+\varepsilon(\delta c)^{-1} C\\
 &>-\frac{c}{c-\bar{u}}-\frac{1}{2} c^2 +\varepsilon(\delta c)^{-1} C\\
& =0\ .
\end{aligned}
$$
This is again a contradiction, and thus $v$ is always bounded as well as $u$, which completes the existence proof.
\end{proof}

As a consequence of Theorem \ref{thrm:existence} equation (\ref{eq:main}) has a unique (up to $\xi$-translations) solution that satisfies the asymptotic conditions (\ref{eq:condition1}) and (\ref{eq:condition2}). 

If we assume purely dissipative solutions ($\delta=0$) then the situation is identical to the dissipative Burger's equation where regularized shock waves are formed \cite{L2013}. On the other hand, when we consider the purely dispersive case ($\varepsilon=0$), then only asymptotically we describe undular bore solutions as dispersive shock waves \cite{EH2016}. In the next section we study the nature of solutions in detail.

\section{Dispersive and dissipative dominated solutions}\label{sec:shape}

Examining the relation (\ref{eq:relation3}) we observe that the stability of the critical point at $-\infty$ changes nature depending on the strength on the dispersion compared to the dissipation. First, recall that 
$$\alpha(c)=\frac{c - \sqrt{ c^2+8}}{2}+\frac{4 c}{(c - \sqrt{ c^2+8})^2}=u_0-c+\frac{c}{(u_0-c)^2}\ ,
$$
and the critical point $(u_0, 0)$ is an unstable spiral when 
$\varepsilon^2<4\delta c\alpha(c)$, while it is an unstable node when $\varepsilon^2\geq 4\delta c\alpha(c)$. The change in type of the critical point signals change in the nature of the traveling wave as well.  In the region  $\varepsilon^2 \ge 4\delta c\alpha(c)$ where dissipation dominates the dispersion,
the solution is not oscillatory but strictly decreasing and resembles a regularized shock wave.

\begin{theorem}\label{thrm:regshock}
Let $u(\xi)$ be the unique solution to (\ref{eq:main}) for values $\varepsilon>0$, $\delta>0$ and $c>1$ such that $\varepsilon^2\geq4 \delta c\alpha(c)$. Then for all $\xi\in \mathbb{R}$, $0<u(\xi)< u_0$ and $u'(\xi)<0$. Moreover, $u$ has a unique inflection point $\xi_1$ such that $(\xi-\xi_1)u''(\xi_1)>0$, for all $\xi\not=\xi_1$.
\end{theorem}
\begin{proof}
Let $(u,v)$ be the corresponding solution to the system (\ref{eq:firstordera}). Since $v=(\delta c)u'$, proving that $u'(\xi)<0$ for all $\xi$ it suffices to prove that for the solution $(u,v)$ we have $u(\xi)>0$ and $v(\xi)<0$ for all $\xi\in \mathbb{R}$, which means that the heteroclinic orbit guaranteed by Theorem \ref{thrm:existence} remains always in the fourth quadrant. To this end, we have already seen that the solution cannot intersect the $v$-axis since $0<u(\xi)<c$ for all $\xi\in\mathbb{R}$. It remains to show that the orbit cannot exit the semi-infinite strip $\mathcal{S}=\{(u,v):v<0,0<u<u_0\}$, which can be established if we show that the orbit is enclosed by a bounded domain of the strip $\mathcal{S}$. In particular, we will show that the orbit remains inside the triangle $\mathcal{T}$ defined by the lines 
$$\ell_0=\{(u,v):v=0, 0\leq u\leq u_0\},~ \ell_1=\left\{(u,v): v= m(u-u_0),0\leq u\leq u_0\right\}\ ,$$
and $$\ell_2=\left\{(u,v):- m u_0 \leq v\leq 0, u=0\right\}\ ,$$
with 
$$m=(\delta c) \Lambda_{-}>0\ .$$ 
Figure \ref{fig:triangle} presents the graph of a heteroclinic orbit $R$ in the dissipative case $\varepsilon^2\geq 4 \delta c\alpha(c)$ enclosed by a triangle $\mathcal{T}$ defined by the edges $\ell_0$, $\ell_1$ and $\ell_2$. The limit of the traveling wave as $\xi\to-\infty$ is depicted by an asterisk.
\begin{figure}[ht!]
\includegraphics[width=\textwidth]{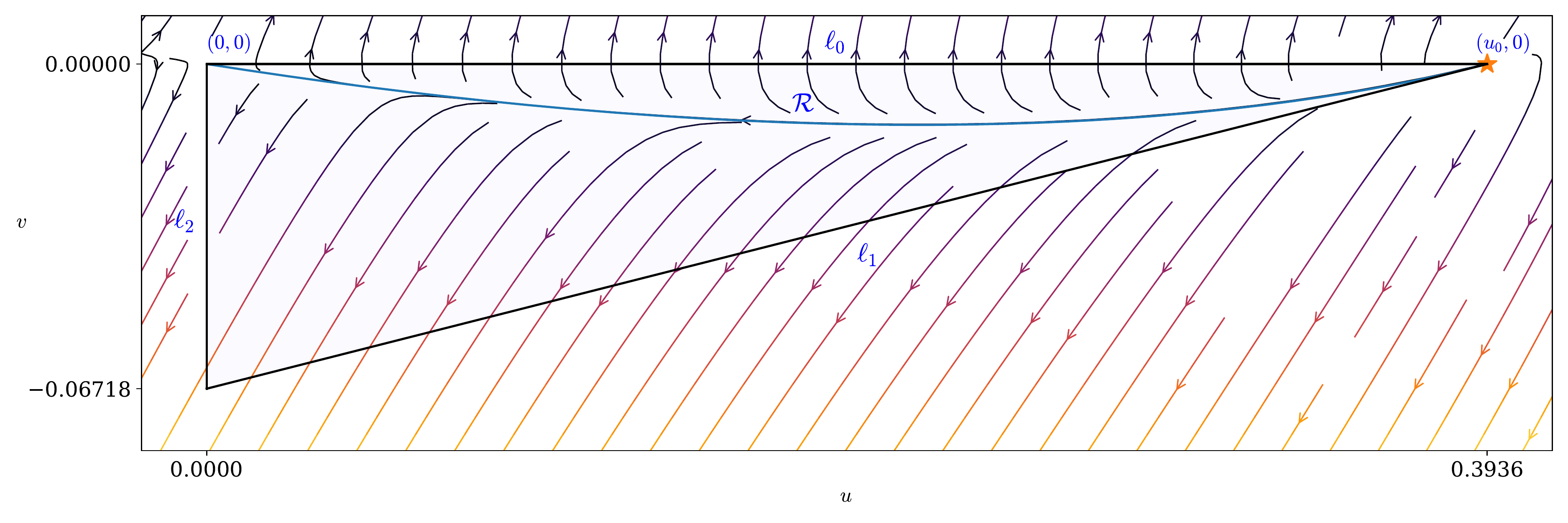}
\caption{Dissipation dominated solution ($c=1.3$, $\delta=0.2$, $\varepsilon=1.2$) enclosed by the triangle $\mathcal{T}$ in the phase-space and the corresponding streamlines of the flow}\label{fig:triangle}
\end{figure}

We have already excluded the possibility of the orbit to intersect the segment $\ell_2$ since $u\not=0$. We show that the orbit $\mathcal{R}$ does not cross the segment $\ell_0=\{(u,v):v=0, 0\leq u\leq u_0\}$.
Since the orbit approaches the origin from the fourth quadrant as $\xi\to+\infty$, assume for contradiction that $\xi_0$ is the (largest) value of $\xi$ for which the orbit intersects the segment $\ell_0$ for the last time in the sense that $v(\xi)<0$ for all $\xi>\xi_0$ and $v(\xi_0)=0$. (For $\xi<\xi_0$ the function $v(\xi)$ can be anything, positive or negative or zero). This means that $v'(\xi_0)\leq 0$. From the second equation of (\ref{eq:firstordera}) we have that
$$v'(\xi_0)=cu(\xi_0)+\frac{u(\xi_0)}{u(\xi_0)-c}-\frac{1}{2}u^2(\xi_0)>0\ ,$$
because we have $0<u(\xi)<u_0$. This is a contradiction, and this means that the orbit doesn't intersect the segment $\ell_0$. 

We will prove that the orbit $\mathcal{R}$ doesn't intersect the segment $\ell_1$ and since the orbit approaches the origin, this means that the whole orbit is confined to the triangle $\mathcal{T}$.

We proceed via contradiction.  We know that the orbit is inside the triangle $\mathcal{T}$ as $\xi\to+\infty$ since it converges to the origin through $Q_4$.
Assume that the orbit intersects the segment $\ell_1$ and that $\xi_0$ is the smallest value such that $(u(\xi_0),v(\xi_0))\in \ell_1$. It follows that $0\leq u(\xi_0)\leq u_0$ and $$(\delta c)u'(\xi_0)=v(\xi_0)= m(u(\xi_0)-u_0)<0\ .$$
Thus, the orbit $\mathcal{R}$ at the point of intersection must have $v'(\xi_0)/u'(\xi_0)\leq m$, because otherwise $u'(\xi_0)>0$. Dividing the equations of the system (\ref{eq:firstordera}), and using the definitions of (\ref{eq:critpo}) we have
$$
\begin{aligned}
m\geq \frac{v'(\xi_0)}{u'(\xi_0)} &=\frac{cu(\xi_0)+\frac{u(\xi_0)}{u(\xi_0)-c}-\frac{1}{2}u^2(\xi_0)+\varepsilon(\delta c)^{-1} v(\xi_0)}{(\delta c)^{-1}v(\xi_0)}\\
&=\varepsilon +u(\xi_0)\frac{c+\frac{1}{u(\xi_0)-c}-\frac{1}{2}u(\xi_0)}{(\delta c)^{-1}v(\xi_0)}\\
 &=\varepsilon   -  \frac{ (\delta c)u(\xi_0)(u_+-u(\xi_0))}{2m(c-u(\xi_0))}  \ .
\end{aligned}
$$
Thus,
$$m^2-\varepsilon m + (\delta c)u(\xi_0)\frac{1}{2}\frac{u_{+}-u(\xi_0)}{c-u(\xi_0)}\geq 0\ .$$
On the other hand 
{\small 
$$
\begin{aligned}
m^2-\varepsilon m + (\delta c)u(\xi_0)\frac{1}{2}\frac{u_{+}-u(\xi_0)}{c-u(\xi_0)} &= m^2-\varepsilon m +(\delta c)\alpha(c)+ (\delta c)\left(\frac{u(\xi_0)}{2}\frac{u_{+}-u(\xi_0)}{c-u(\xi_0)}-\alpha(c)\right)\\
&=(\delta c)\left(\frac{u(\xi_0)}{2}\frac{u_{+}-u(\xi_0)}{c-u(\xi_0)}-\alpha(c)\right)\\
&<0\ .
\end{aligned}
$$
}
To see this, it suffices to prove that  $u(u_+-u)<2\alpha(c)(c-u)$ for all $0\leq u< u_0$, or equivalently that
$$u^2-(2\alpha+u_+)u+2\alpha(c) c>0 \ .$$
For this, define $D(u)=u^2-(2\alpha+u_+)u+2\alpha(c) c$. Then we have
$$\frac{d}{du}D(u)=2u-(2\alpha(c)+u_+)\ .$$
Since $u_0+u_+=3c$ and $\alpha(c)=u_0-c+c/(u_0-c)^2$, we have
$$
\begin{aligned}
\frac{d}{du}D(u_0)&=2u_0-2\alpha(c)-u_+\\
&=2u_0-2\left(u_0-c+\frac{c}{(u_0-c)^2}\right)-3c+u_0\\
&=(u_0-c)-\frac{2c}{(u_0-c)^2}<0\ .
\end{aligned}
$$
Moreover,
$$\begin{aligned}
D(u_0)&=u_0^2-(2\alpha(c)+3c-u_0)u_0+2\alpha(c) c\\
&=-\frac{c}{c-u_0}[u_0^2-3cu_0+2(c^2-1)]\\
&=0\ ,
\end{aligned}
$$
where for the last equality please see the proof of Lemma \ref{lem:asymptot1}.

Since $D(u)$ is decreasing, we conclude that $D(u)>D(u_0)=0$ for $u<u_0$. Therefore, for $c>1$, and since $u(\xi_0)<u_0$ we get that  
$$m^2-\varepsilon m + (\delta c)u(\xi_0)\frac{u_{+}-u(\xi_0)}{c-u(\xi_0)}<0\ ,$$
which is a contradiction.

The fact that the orbit remains into the triangle $\mathcal{T}$ can be implied by the directionality of the streamlines in Figure \ref{fig:triangle}, which ensures that if the orbit $\mathcal{R}$ exits from one of the sides of the triangle, then it will never be able to return and asymptotically converge to the origin from the fourth quadrant. The study of the flow through the edges of the triangle $\mathcal{T}$ follows the exact same calculations presented in this proof.

Thus, $u(\xi)$ decreases monotonically from $u_0$ to $0$ as $\xi$ increases. Since $v(\xi)<0$ for all $\xi$ and $v(\xi) \to 0$ as $\xi \to \pm \infty$, the intermediate
value theorem guarantess there is $\xi_1\in\mathbb{R}$ such that  $v'(\xi_1)=u''(\xi_1)=0$. 
If $v'(\xi)=0$, then from the second equation of the system (\ref{eq:firstordera}) we observe that
$$v''(\xi)=-u'(\xi)\frac{c+(u(\xi)-c)^3}{(u(\xi)-c)^2}\ .$$
Since $0<u(\xi)<u_0<c$ and $u'(\xi)<0$ for all $\xi$, we have the following possibilities:
\begin{enumerate}
\item [(i)] If $0<u(\xi)<c-\sqrt[3]{c}$ then $v''(\xi)<0$ and $\xi_1$ is a strict local maximum for $v$
\item [(ii)] If $c-\sqrt[3]{c}<u(\xi)<u_0$ then $v''(\xi)>0$ and $\xi_1$ is a strict local minimum for $v$
\item [(iii)] If $u(\xi)=c-\sqrt[3]{c}$ then $v''(\xi)=0$ and $\xi_1$ is a saddle point for $v$
\end{enumerate}
It is worth mentioning that $c-\sqrt[3]{c}<u_0$ for $c>1$. Thus, for the local minimum $\xi_1$ we have that $c-\sqrt[3]{c}<u(\xi_1)<u_0$. Assume for contradiction that there is another $\xi_1'\not=\xi_1$ local minimum for $v$. Without loss of generality assume that $\xi_1'<\xi_1$. By the mean value theorem we have again that there is a local maximum $\bar{\xi}$ with $\xi_1'
<\bar{\xi}<\xi_1$. Since $u$ is strictly decreasing, we have that $u(\xi_1')>u(\xi)>u(\xi_1)>c-\sqrt[3]{c}$ which contradicts (ii). Therefore, if for some $\xi\not=\xi_1$ we have $v'(\xi)=0$, then  $0<u(\xi)\leq c-\sqrt[3]{c}$ and $\xi_1<\xi$. If  $u(\xi)<c-\sqrt[3]{c}$ this means that $\xi$ is strictly local maximum, and also $v(\xi)<0$. Since $v\to 0$ as $\xi\to +\infty$ this means that there will be another local minimum for $\xi<\xi_1''$ which is a contradiction to (i). Finally, if $v'(\xi)=0$ and $u(\xi)=c-\sqrt[3]{c}$, then from the second  equation of (\ref{eq:firstordera}) we have that
$$ v(\xi)=-\frac{\delta c}{\varepsilon}\left(2-3\sqrt[3]{c^2}+c^2\right)<0, \quad\text{ for $c>1$}\ .$$
In the proof of Theorem \ref{thrm:existence} we saw that $v(\xi)\geq -(\delta c)(2-3\sqrt[3]{c^2}+c^2)/\varepsilon$, and thus $\xi$ cannot be a saddle point. Hence $\xi_1$ is the only point where $v'(\xi_1)=u''(\xi_1)=0$ and $c-\sqrt[3]{c}<u(\xi_1)<u_0$. Moreover, $v'(\xi)$ has one sign for $\xi>\xi_1$ and similarly for $\xi<\xi_0$. Since $v(\xi_1)<0$ and $v(\xi)\to 0$ as $\xi\to+\infty$, then for $\xi>\xi_1$ we have $v'(\xi)>0$. Similarly, for $\xi<\xi_1$ we have $v'(\xi)<0$. Thus, we conclude that there is a unique inflection point $\xi_1$ of $u$ such that $(\xi-\xi_1)u''(\xi)>0$ for $\xi\not=\xi_1$, and the proof is complete.
\end{proof}

Since $\eta' = \frac{c}{(u-c)^2} u'$, the sign $\eta'$ follows the sign of $u'$. We turn now to the complementary regime $\varepsilon^2<4\delta c\alpha(c)$, where dispersion dominates dissipation. We call this the regime of
moderate dispersion, and note that there the solution resembles to an undular bore comprised of a smooth wave front followed by a series of oscillations that decay
to nil as $\xi \to -\infty$: 

\begin{theorem}\label{thrm:undularbore}
Let $u(\xi)$ be the unique (up to horizontal translations) solution to (\ref{eq:main}) for values $\varepsilon$, $\delta$ and $c$ such that $\varepsilon^2<4\delta c\alpha(c)$. Then for all $\xi\in\mathbb{R}$, $u(\xi)>0$, and 
\begin{itemize}
\item [(a)] If $M_0=\sup_\xi u(\xi)$, then $M_0$ is attained at a unique value $\xi=z_0$, and for $\xi>z_0$, $u'(\xi)<0$.
\item [(b)] There is a $\xi_0>z_0$ such that $(\xi-\xi_0)u''(\xi)>0$, for all $\xi>z_0$, $\xi\not=\xi_0$.
\item [(c)] The solution $u(\xi)$ has an infinite number of local maxima and minima. These are taken on at points $\{z_i\}_{i=0}^\infty$ and $\{w_i\}_{i=1}^\infty$ where $z_i>w_{i+1}>z_{i+1}$, for all $i\geq 0$, and $\lim_{i\to\infty}z_i=\lim_{i\to\infty}w_i=-\infty$. Moreover, 
$u_0<u(z_{i+1})<u(z_i)$ and  $u_0>u(w_{i+1})>u(w_i)$.
\end{itemize}
\end{theorem}
\begin{proof}
Since $(u_0,0)$ is a spiral point in the case of $\varepsilon^2<4\delta c\alpha(c)$ we expect that the solution will be oscillatory as $\xi\to -\infty$ and more specifically of the form
\begin{equation}\label{eq:harman}
(u(\xi)-u_0,v(\xi))=Ce^{\mu \xi}(\cos(\nu\xi+\theta_0+o(1)),\sin(\nu\xi+\theta_0+o(1)))\ ,
\end{equation}
as $\xi\to-\infty$, where $\mu={\rm Re} (\Lambda_+)$ and $\nu={\rm Im}(\Lambda_+)\not=0$, with $\Lambda_+$ as defined in (\ref{eq:eigen2}), while $C$ and $\theta_0$ depend on the solution \cite{Hartman2002}. Therefore, there are infinite many points where $v(\xi)$ vanishes. According to Lemma \ref{lem:critical} these points will be isolated and are strict local maxima or minima of $u$. Moreover, a local maximum corresponds to a point $\xi$ where $u(\xi)>u_0$, and a local minimum to a point $\xi$ where $0<u(\xi)<u_0$. 

Since $(u,v)\to(0,0)$ as $\xi\to\infty$ and $u>0$ for all $\xi$ we have that $v(\xi)<0$ for $\xi\to\infty$ and thus the orbit $\mathcal{R}$ lies in the fourth quadrant for large values of $\xi$, while the orbit cannot intersect the segment $\ell_0=\{(u,v):v=0,~ 0\leq u\leq u_0\}$ from $Q_4$. To see this (as in the proof of Theorem \ref{thrm:regshock}) assume that there is a point $\xi_0$ such that $(u(\xi_0),v(\xi_0))\in \ell_0$, i.e. $v(\xi_0)=0$ and $u(\xi_0)\leq u_0$, and $v(\xi)<0$ for $\xi>\xi_0$. This means that $v'(\xi_0)\leq 0$, and from the second equation of (\ref{eq:firstordera}) we have that
$$v'(\xi_0)=cu(\xi_0)+\frac{u(\xi_0)}{u(\xi_0)-c}-\frac{1}{2}u^2(\xi_0)>0\ ,$$
since $0< u(\xi)< u_0$, which is a contradiction. Thus $\mathcal{R}$ must exit fourth quadrant at a point $z_0$ where $u(z_0)>u_0$ and $v(z_0)=0$. And $z_0$ will be a strict local maximum of $u$ and $u'(\xi)<0$ for $\xi>z_0$. Following the same steps as in the proof of Theorem \ref{thrm:regshock} we conclude that there is a point $\xi_0>z_0$ where $u''(\xi_0)=0$ and $u''(\xi)>0$ for $\xi>\xi_0$, and $u''(\xi)<0$ for $\xi<\xi_0$.

Let $z_1<z_0$ be the next local maximum of $u$ and let $w_1$ be the unique intervening local minimum. Inductively we define the decreasing sequences $\{z_i\}_{i=0}^\infty$ and $\{w_i\}_{i=1}^\infty$ of local maxima and minima, respectively. Due to Lemma \ref{lem:critical} we have that $u_0<u(z_i)<\bar{u}$ and $u(w_i)<u_0$. Moreover, $u(z_{i+1})<u(z_i)$ and $u(w_{i+1})>u(w_i)$ for all $i$ because otherwise the orbit $\mathcal{R}$ will intersect itself which is impossible since system (\ref{eq:firstordera}) is autonomous, \cite{W2003}. For the same reason, $w_i$ and $z_i$ must accumulate at $-\infty$ and the solution cannot become constant $u_0$ in finite time.
\end{proof}

Figure \ref{fig:dsw1} shows that as the speed of the oscillatory shock wave increases, the solution tends to become weakly singular \cite{MDAZ2017,DMM2018}. This is true also for the regularized shock wave as it is clear in Figure \ref{fig:dsw1}. Figure \ref{fig:dsw2} shows that as the speed of the traveling wave increases, the value $u_0$ of the tail converges to the maximum amplitude of the traveling wave. The limit $u_0$ of the tail is depicted with a star in Figure \ref{fig:dsw2}.
The numerical runs indicate that as the traveling wave speed increases  the tops of the free surface elevation become very steep. This should be
contrasted to the behavior of traveling waves for the viscous KdV equation.

\begin{figure}[ht!]
\includegraphics[width=\textwidth]{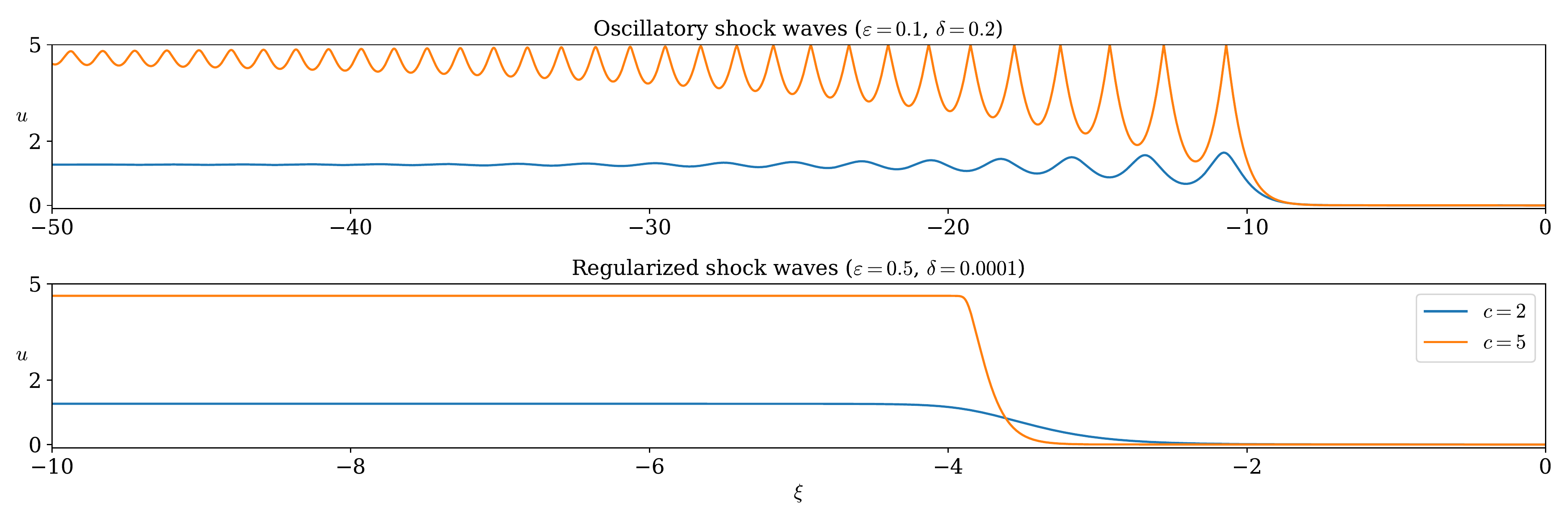}
\caption{Oscillatory and regularized shock waves of the dissipative Boussinesq system for $c=2$ and $c=5$}\label{fig:dsw1}
\end{figure}

\begin{figure}[ht!]
\includegraphics[width=\textwidth]{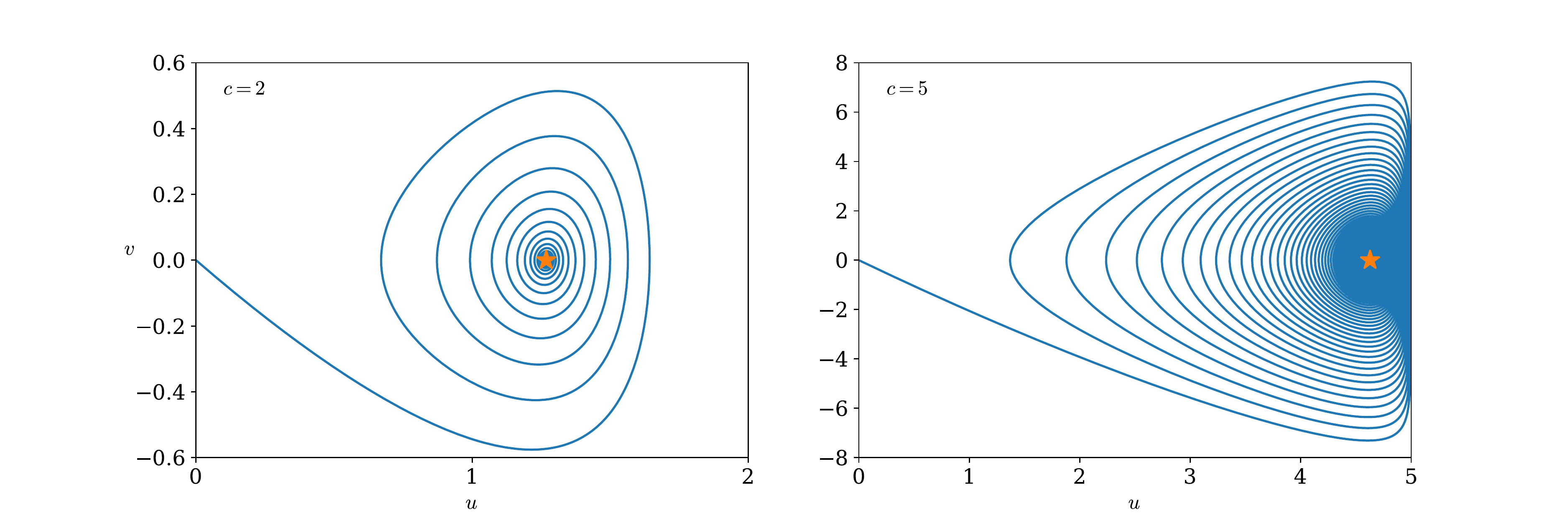}
\caption{Heteroclinic orbits representing diffusive-dispersive shock waves for $c=2$ and $c=5$}\label{fig:dsw2}
\end{figure}

%

As it is demonstrated so far, the existence and the shape of an undular bore as a traveling wave is related to the amount of dissipation in the flow $\varepsilon$ and the depth of the water through the parameter $\delta$. We close this section with a comparison between laboratory data and the predictions of the current study.

In laboratory experiments (\cite{T1994,KC2006,Chanson2010}) it has been observed that for Froude numbers $1<c\leq 1.3$ the turbulence in the flow can be minimal, and thus it is expected that dissipation due to bottom friction or surface tension can only affect slightly the shape of an undular bore. On the other hand, for Froude number $c>1.3$ large turbulence is present in the flow, especially affecting the wave-front, while for $c>1.4$ wave breaking can occur at large scales.

We begin with comparing the diffusive-dispersive shock wave of (\ref{eq:boussinesqs2}) with a smooth undular bore where the dissipation should not be important. In Figure \ref{fig:exper2} we compare an oscillatory shock wave for Froude number $c=1.11$ and dissipation parameter $\varepsilon=0.06$ of (\ref{eq:boussinesqs2}) with a laboratory experiment of \cite{Chanson2010}. The independent and dependent variables in Figure \ref{fig:exper2}, $t$ and $\eta$, are scaled and coincide with the standard time variable scaled by $\sqrt{g/D}$ and surface elevation scaled by $D$.  Although there is no wave-breaking or turbulence effects in the particular case, dissipation can be caused due to bottom friction or even due to surface tension effects, thus the presence of the dissipative term can help in obtaining perfect agreement between laboratory data and theoretical estimates.

\begin{figure}[ht!]
\includegraphics[width=\textwidth]{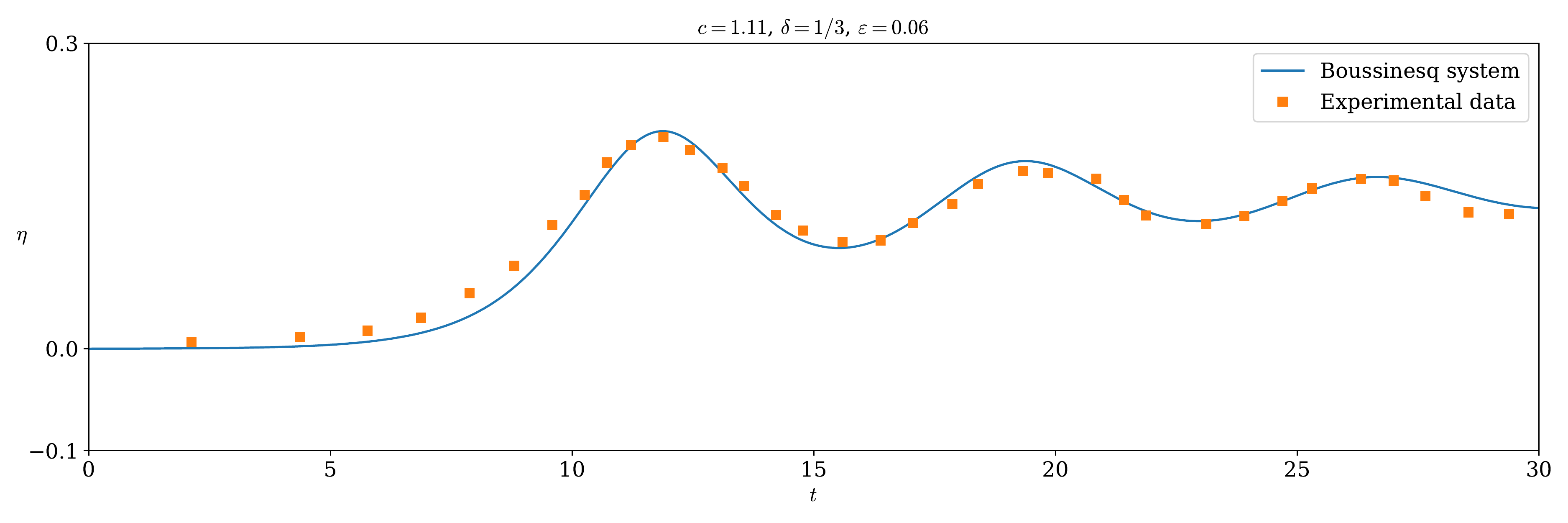}
\caption{Comparison between a diffusive-dispersive shock wave of system (\ref{eq:boussinesqs2}) and experimental data of \cite{Chanson2010}}\label{fig:exper2}
\end{figure}

\begin{figure}[ht!]
\includegraphics[width=\textwidth]{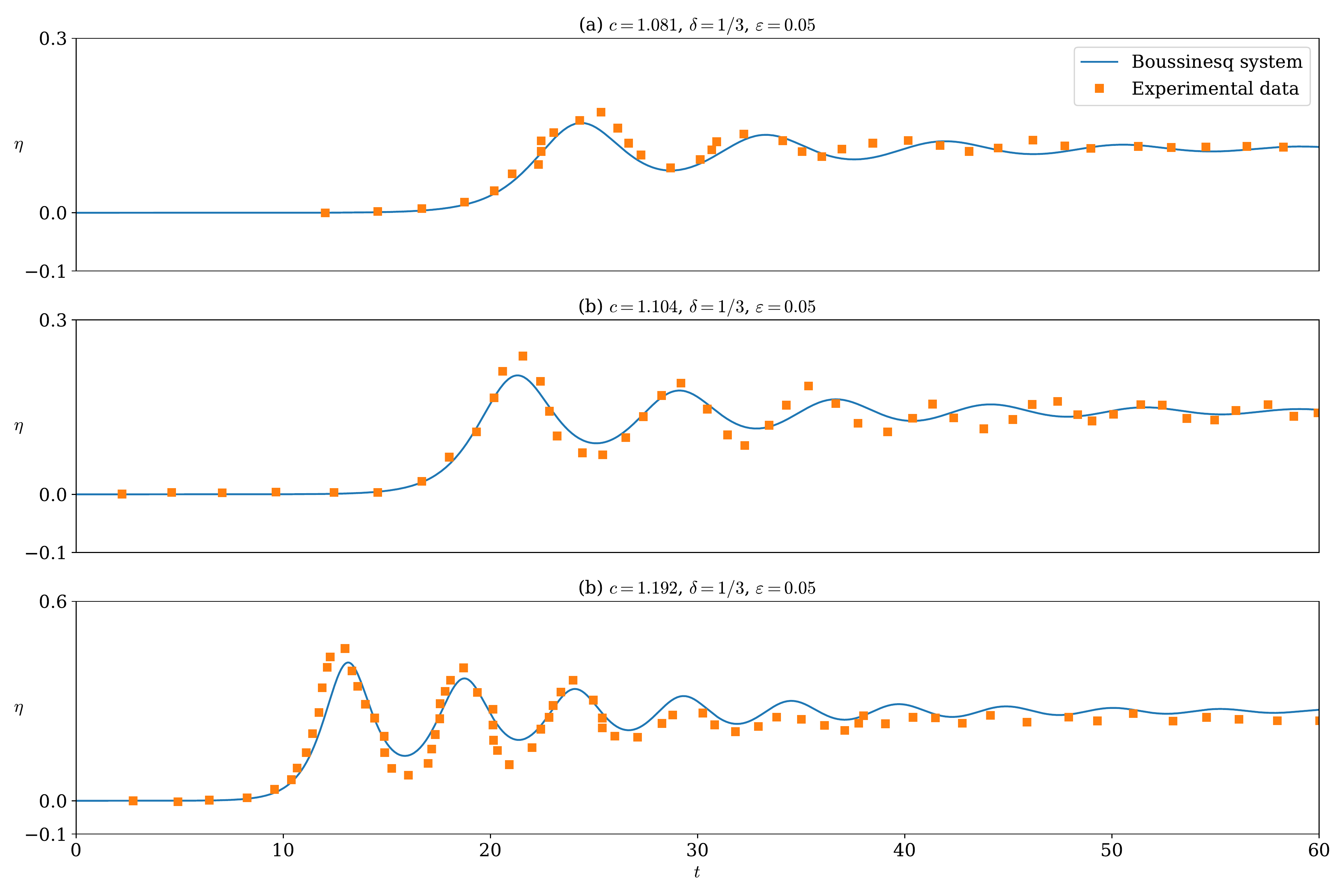}
\caption{Comparison between diffusive-dispersive shock waves of system (\ref{eq:boussinesqs2}) and experimental data of \cite{SZ2002}}\label{fig:exper1}
\end{figure}

In the same spirit, we compare the solutions obtained from system (\ref{eq:boussinesqs2}) with  experimentally generated undular bores with Froude numbers $c=1.081$, $1.104$ and $1.192$, respectively, \cite{SZ2002}. These waves propagate without breaking and with clear undulations. As it is observed also in \cite{T1994}, turbulence becomes visible for values of $c>1.2$. For this reason the dissipation parameter was taken $\varepsilon=0.05$ in all cases, while similar values of $\varepsilon$ will not cause major differences. Figure \ref{fig:exper1} presents the results in the same scaled variables as before. Discrepancies that were expected to the phase of the modulations due to the suboptimal linear dispersion characteristics of the Boussinesq system at hand are small.

 \begin{figure}[ht!]
\includegraphics[width=\textwidth]{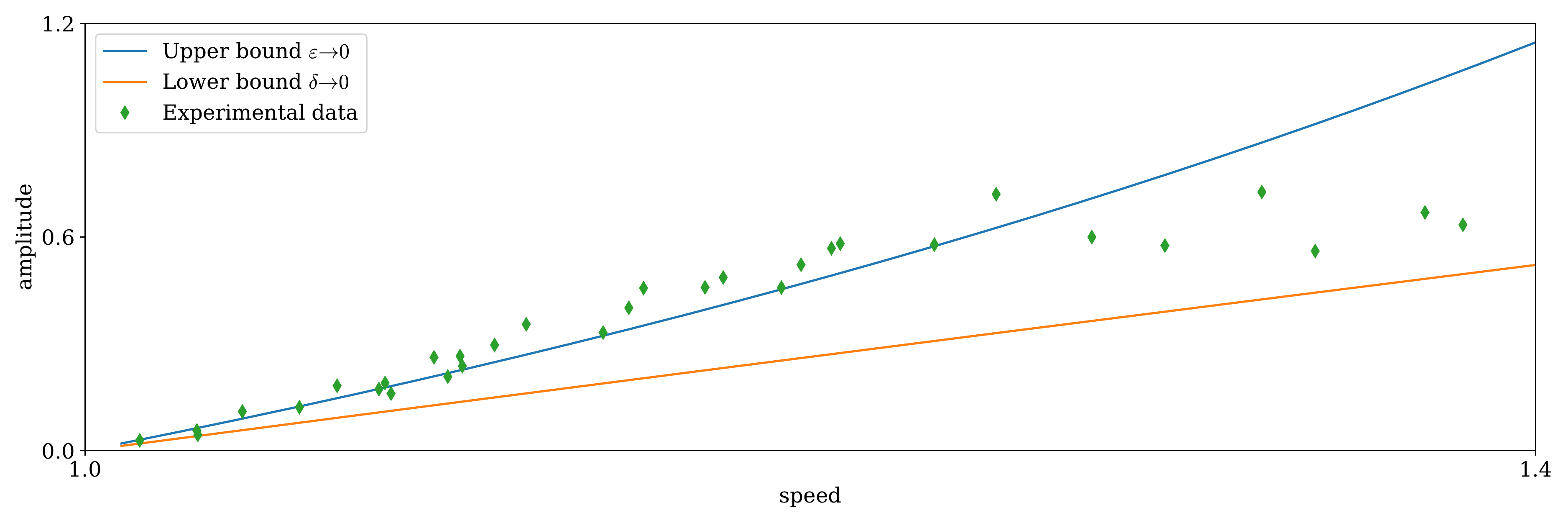}
\caption{The amplitude of diffusive-dispersive shock wave as a function of its speed $c$ in comparison with experimental data of \cite{Favre1935}, \cite{SZ2002} and \cite{T1994}}\label{fig:speedampl}
\end{figure}

Apparently, the Boussinesq approximation predicts the shape of the undular bore quite accurately. Figure \ref{fig:speedampl} presents the possible extreme values for the amplitude of a diffusive-dispersive shock wave as function of its speed. The maximum possible value of the amplitude is numerically generated by solving using Newton's method the speed-amplitude relationship for solitary waves (\ref{eq:spamp}) with respect to $\eta$ for various values of $c\in[1,1.4]$. Such an extreme amplitude can be obtained when $\varepsilon \to 0$ and $\delta$ remains fixed. The lower bound is obtained when $\delta\to 0$ while $\varepsilon$ remains fixed and is the value for the limit $\eta_0$, where $\eta_0=\frac{u_0}{c-u_0}$ and $u_0$ as in Lemma \ref{lem:asymptot1}. Note that this value is independent of $\varepsilon$. An alternative approximation of the amplitude is  given by the explicit approximation \cite{T1994} 
$$c=\sqrt{1+\frac{3}{2}\eta_0+\frac{1}{2}\eta_0^2}\ .$$
In Figure \ref{fig:speedampl} we observe that the maximum values recorded in the experiments of \cite{Favre1935}, \cite{SZ2002} and \cite{T1994} are all closer to the Boussinesq approximation (\ref{eq:spamp}) for the amplitude of the solitary wave, and for increasing values of the speed, the dissipation due to turbulence becomes more important leading to a dominant dissipation parameter. Specifically, the speed-amplitude line follows the non-dissipative relationship for Froude numbers less than $c\approx 1.25$, while it reaches its lowest dissipative branch for Froude number $c\approx 1.4$. These results validate also the asymptotic theory for non-dissipative systems which states that the leading wave of a dispersive shock wave is a solitary wave \cite{EH2016}. Figure \ref{fig:break} presents a comparison between laboratory data of \cite{LC2015} with the profile of a weakly oscillatory shock wave (almost regularized shock wave) with Froude number $c=1.45$. Wave breaking causes dissipation due to turbulence and we found with experimentation that for the value $\varepsilon=0.6$ the solution of the dissipative Peregrine system fits well to the experimental data. For larger values of speed $c$ and when more severe wave-breaking is present then the shape of the wave cannot be predicted by the particular dissipative model and more sophisticated wave-breaking  techniques must be incorporated in numerical simulations \cite{CC2021}.

 \begin{figure}[ht!]
\includegraphics[width=\textwidth]{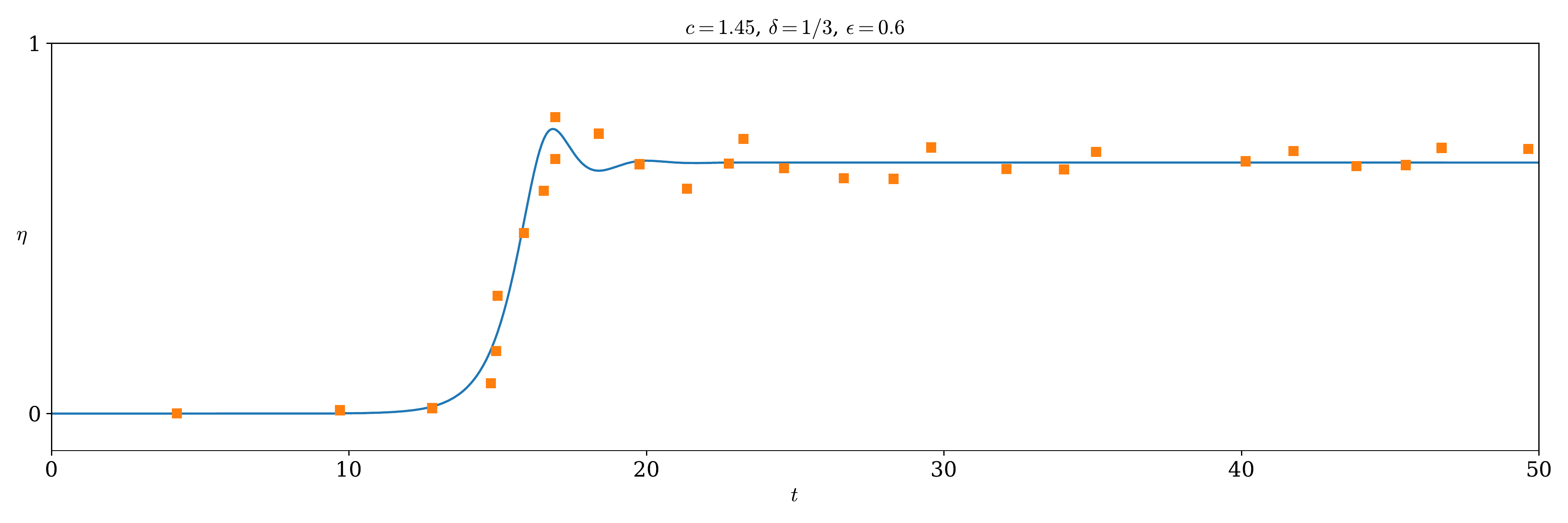}
\caption{Weakly oscillatory shock wave - Breaking undular bore}\label{fig:break}
\end{figure}

It is noted that all the profiles were generated numerically by solving the system (\ref{eq:firstordera}) with appropriate initial conditions using the implicit Runge-Kutta method of the Radau IIA family of order 5 as it is implemented in the Python module \texttt{scipy.integrate}.  

\begin{remark}
In \cite{BMT2023} it is proved that as $\varepsilon\to 0$ and $\delta=o(\varepsilon)$, the diffusive-dispersive shock wave tends to a classical shock wave of the shallow-water equations. This justifies also the use of the term shock wave for traveling wave solutions of diffusive-dispersive nonlinear equations. 
\end{remark}

 
\section{Undular bores in the absence of dissipation}\label{sec:errors}

Although non-dissipative systems such as Peregrine's system (\ref{eq:peregrine}) are not known to possess oscillatory or regularized shock waves as traveling wave solutions,
they are often used for the description of such waves. The reason being that solutions often exhibit the format of undular bore consisting of a solitary wave front followed by undulations. Here, we prove an error estimate between the solutions of the system  (\ref{eq:boussinesqs2}) for $\varepsilon>0$ and for $\varepsilon=0$, where the last one corresponds to Peregrine's system (\ref{eq:peregrine}). 
For a study of the long-time behavior of the solutions of dissipative Boussinesq systems we refer to \cite{CG2007}.

We will use two lemmata from \cite{AB1991} which we mention here for the sake of completeness. The first lemma provides a useful inequality:
\begin{lemma}{(Lemma 1 of \cite{AB1991})}\label{lem:lemmab1}
Let $k\geq 3$ be an integer. Then there exists a constant $C=C(k)>0$ such that for any $y>0$ and $\epsilon$ in the interval $(0,1)$, the inequality
\begin{equation}
\epsilon y+y^2+\epsilon^{-1/2}y^3+\epsilon^{-1}y^4+\cdots +\epsilon^{(2-k)/2}y^k\leq C(\epsilon y+\epsilon^{2-k}y^k)\ ,
\end{equation}
is valid.
\end{lemma}
Here we will use this lemma for $k=3$. The second lemma, again adapted to our needs, is a Gronwall's type inequality.
\begin{lemma}{(Lemma 2 of \cite{AB1991})}\label{lem:lemmab2}
Let $A, B>0$ be given. Then there is a maximal time $T$ and a constant $C>0$ independent of $A$ and $B$ such that if $y(t)$ is any non-negative differentiable function defined on $[0,T]$ satisfying
$$\begin{aligned}
&\frac{d}{dt}y^2(t)\leq A y(t) + B y^3(t),\qquad 0\leq t\leq T\ ,\\
&y(0)=0\ ,
\end{aligned}$$
then 
$$y(t)\leq CA t, \quad \text{ for }\quad 0\leq t\leq T\ .$$
\end{lemma}
Let $f(x)$, $g(x)$ sufficiently smooth functions such that the initial value problem  (\ref{eq:boussinesqs2}) with initial data $(\eta^\varepsilon(x,0),u^\varepsilon(x,0))=(f(x),g(x))$ has a unique solution $(\eta^\varepsilon,u^\varepsilon)\in C^1([0,T],H^p\times H^q)$ with $p=1,q=2$ for $\varepsilon>0$, and $p=q=2$ for $\varepsilon=0$, at least for times $t\in [0,T]$ for $T>0$. In the previous notation $H^s$ denotes the usual Sobolev space of weakly differentiable functions of order $s$ with associated norm $\|\cdot\|_s$. The usual $L^2$ space is the space $H^0$ with usual norm $\|\cdot\|$. Then for $\varepsilon>0$ the function $(h,w)=(\eta^\varepsilon-\eta^0,u^\varepsilon-u^0)$ satisfies the system
\begin{align}
&h_t+w_x+(hu^\varepsilon+\eta^0 w)_x=0\ , \label{eq:erreq1}\\
& w_t+h_x+u^\varepsilon w_x+w u^0_x-\delta w_{xxt}=\varepsilon u^\varepsilon_{xx}\ . \label{eq:erreq2}
\end{align}
Under the prescribed assumptions for the solutions of system (\ref{eq:boussinesqs2}), we have the following:
\begin{theorem}\label{thrm:last}
The error functions $(h,w)$ corresponding to the solutions of system (\ref{eq:boussinesqs2}) for $\varepsilon>0$ and $\varepsilon=0$ are such that 
$$\sqrt{\|h\|^2+\|w\|^2+\|w_x\|^2}=O(\varepsilon t)\ ,$$
for all $t\in [0,T]$.
\end{theorem}

\begin{proof}
In this proof we will write $\lesssim$ to denote $\leq C$ where $C$ is a positive constant independent of $\varepsilon$.
If we denote 
$$y(t)=\sqrt{\|h\|^2+\|w\|^2+\|w_x\|^2}\ ,$$ 
then, $y(0)=0$ since we assumed the same initial data for both systems. Multiply (\ref{eq:erreq1}) with $h$ and integrate over $\mathbb{R}$ to obtain
$$
\begin{aligned}
\frac{d}{dt}\left(\frac{1}{2}\int_{-\infty}^\infty h^2~dx\right)&=-\int_{-\infty}^\infty w_x h +\frac{1}{2}h^2u^\varepsilon_x+(\eta^0 w)_x h ~dx\\
&\lesssim \|w_x\|~\|h\|+\|h\|^2\|u^\varepsilon_x\|_\infty+\|\eta^0_x\|_\infty\|w\|\|h\|+\|\eta^0\|_\infty\|w_x\| \|h\|\\
&\lesssim \|w\|^2+ \delta \|w_x\|^2+\|h\|^2\ .
\end{aligned}
$$
Similarly we have
{\small
$$
\begin{aligned}
\frac{d}{dt}\left(\frac{1}{2}\int_{-\infty}^\infty w^2+\delta w_x^2  ~dx \right)&= \int_{-\infty}^\infty hw_x-hww_x-u_xw^2~dx-\varepsilon \int_{-\infty}^\infty u_xw_x~ dx\\
&\lesssim \|h\|\|w_x\|+\|h\|\|w\|_\infty\|w_x\|+\|u_x\|_\infty\|w\|^2+\varepsilon \|u_x\|_\infty\|w_x\|\\
&\lesssim \|h\|^2 +\|w\|^2 +\delta \|w_x\|^2 + \varepsilon \|w_x\|\ .
\end{aligned}
$$ 
}
The previous inequalities imply
$$\frac{d}{dt} y^2(t)\lesssim\varepsilon y(t)+y^2(t) \lesssim \varepsilon y(t)+\frac{1}{\varepsilon}y^3(t)\ ,$$
where the last inequality is proved in Lemma \ref{lem:lemmab1} (Lemma 1 of \cite{AB1991}). Then using the Gronwall type inequality from Lemma \ref{lem:lemmab2} (Lemma 2 of \cite{AB1991}) we get
$$y(t)\lesssim \varepsilon t\ .$$ 
Note that the various constants $C$ in the previous inequalities are independent of $\varepsilon$ but depend only on $\delta$ and on bounds of the solutions of the two systems. 
\end{proof}

We illustrate this theorem with a set of numerical experiments comparing the dispersive shock with oscillatory shocks of the dissipative Boussinesq system with $\delta=1$ and $\varepsilon=0.1$ and $0.01$, respectively, for the same initial conditions. We consider initial data  $\eta(x,0)$ to be a smooth approximation of Riemann data (see Figure \ref{fig:comparison} for time $t=0$)  and $u(x,0)=0$. For completeness we present the corresponding solution of the nonlinear shallow water equations with $\varepsilon=\delta=0$. 
\begin{figure}[ht!]
\includegraphics[width=\textwidth]{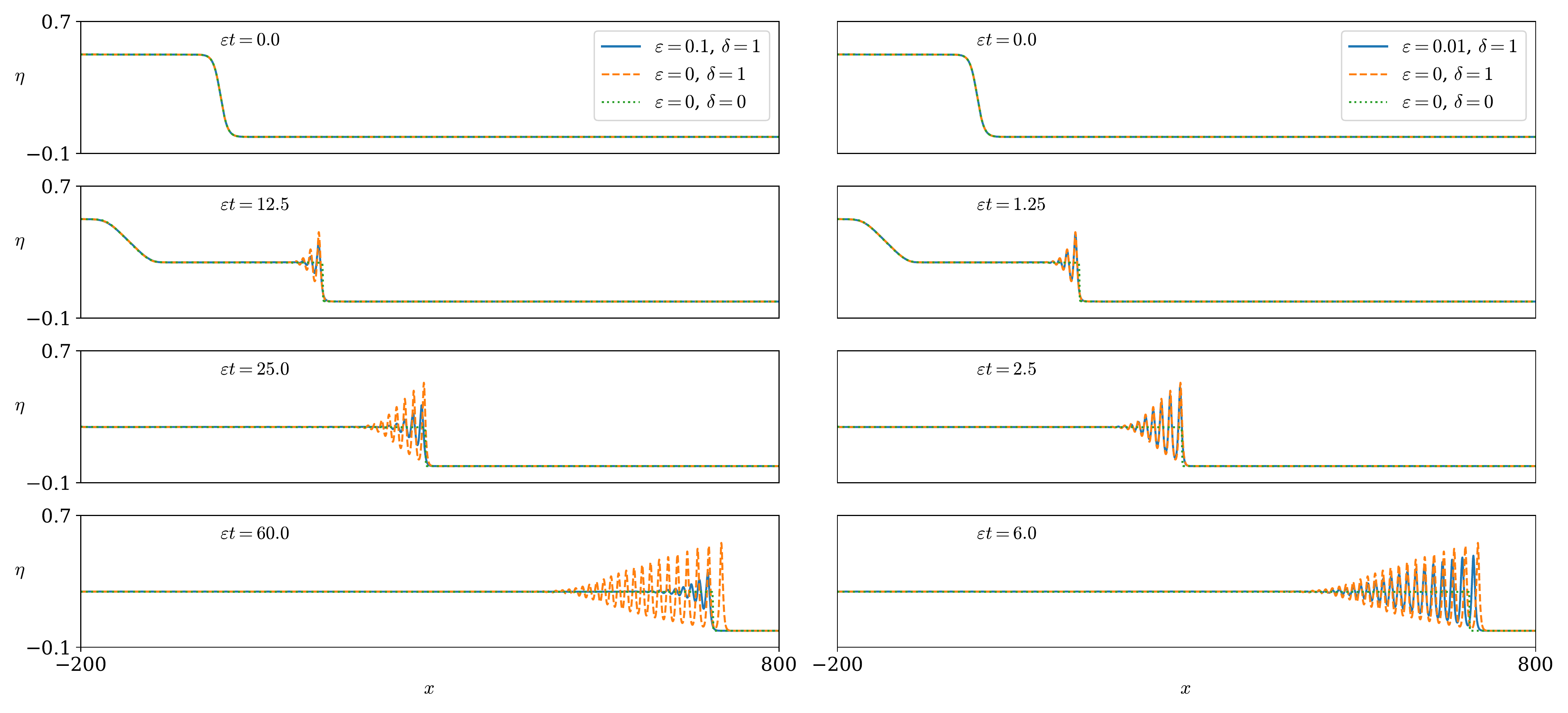}
\caption{Oscillatory shock waves vs dispersive shock waves}\label{fig:comparison}
\end{figure}
\begin{figure}[ht!]
\includegraphics[width=\textwidth]{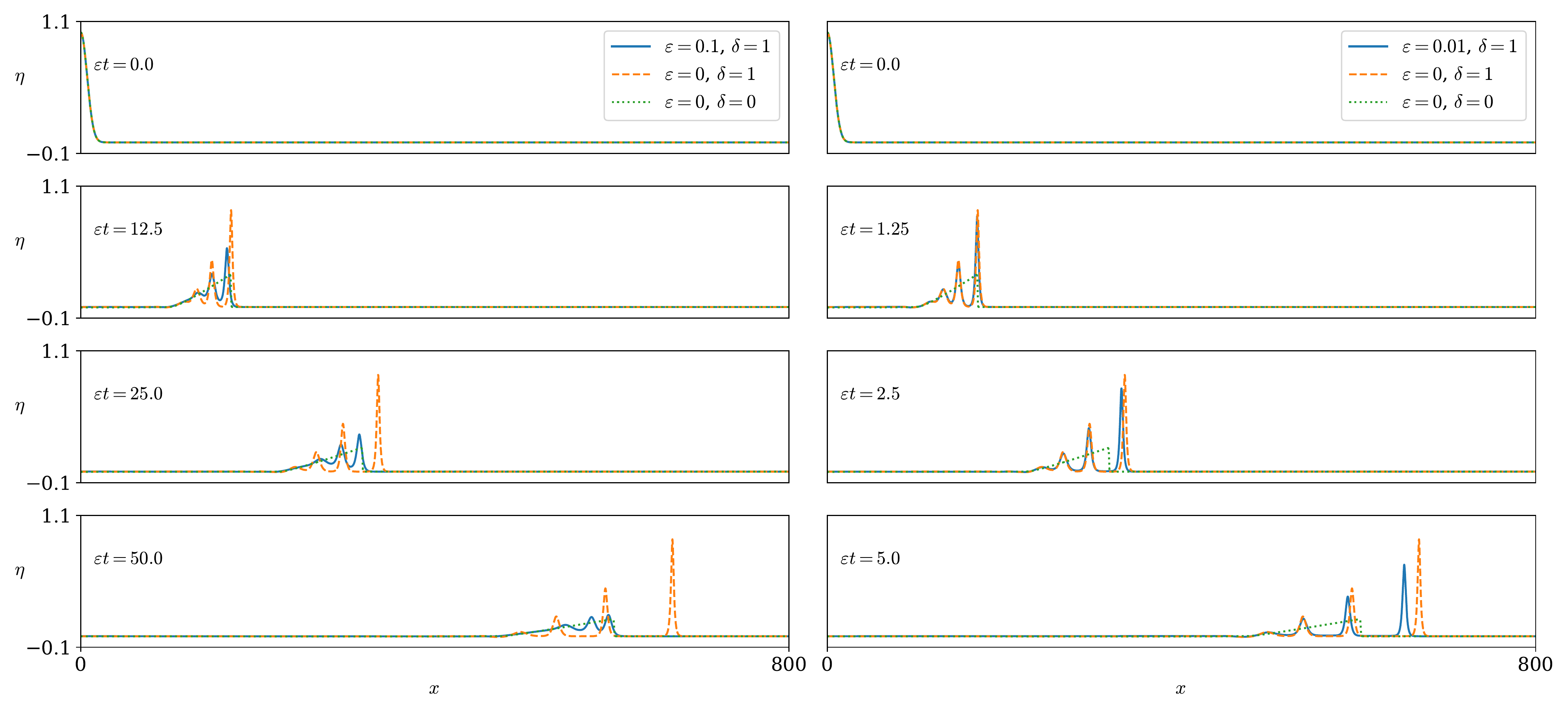}
\caption{Comparison between dissipative and non-dissipative Boussinesq system}\label{fig:comparison2}
\end{figure}
These data generate two expanding waves that propagate in different directions: An oscillatory wave that propagates to the right and a rarefaction wave propagating to the left. When $\varepsilon>0$ the right-traveling waves tend to become oscillatory shock waves while the left-traveling comprised indistinguishable rarefaction waves. The error between the oscillatory shock waves and the dispersive shock waves emanating from the same initial conditions when $\varepsilon=0$ is $O(\varepsilon t)$. This is demonstrated in Figure \ref{fig:comparison} where we observe that for $\varepsilon=0.01$ the two solutions are closer for larger times. For example, it is hard to observe differences between the two solutions at $\varepsilon t=2.5$, while large differences can be observed at $\varepsilon t=6$. On the other hand, when $\varepsilon=0.1$ the two solutions start diverging at an earlier time $t$. We also observe that the classical shock wave of the nonlinear shallow water equations is formed at about the same time with the oscillatory shock of $\varepsilon=0.1$ while  the second oscillatory shock wave needs additional time to get formed and thus leads the classical shock wave. For the numerical solution of the Boussinesq system we considered the numerical method of \cite{AD2012, ADM2010ii} while for the shallow water equations we used the numerical method of \cite{dkm2011}. In all cases we used $\Delta x= 0.25$ and $\Delta t=0.025$ for stepsizes of the numerical methods. The interval we used was $[-800,800]$ but here we only present the results in $[-200,800]$.

We close this section with a numerical  experiment that describes the generation of series of waves from a Gaussian initial condition. Specifically, we consider the same systems as before  but with initial data $\eta(x,0)=e^{-x^2/100}$, $u(x,0)=0$. It is known that in the absence of dissipation and from such initial conditions a series of solitary waves will be generated followed by dispersive tails. Figure \ref{fig:comparison2} shows the solution that propagates only to the right in the interval $[0,800]$ while the solution is symmetric in $[-800,800]$.  As we observe in Figure \ref{fig:comparison2}, the resulting dissipative solutions do not comprised traveling waves anymore but diffusive pulses followed by diffusive dispersive tails. The error estimate again predicts the correct behavior of the various dissipative systems, which is inherited by the stable character of the non dissipative solitary waves. This experiment indicates that traveling waves of system (\ref{eq:boussinesqs2}) can be generated mainly by Riemann initial data. On the other hand, shock waves of the nonlinear shallow water equations  can be generated even by localized initial conditions like the Gaussian one used here.

\section{Conclusions}

Peregrine's system was originally derived to describe undular bores (also known as positive surges) as well as classical solitary waves. Although Peregrine's system along with other Boussinesq type systems describe bidirectional  propagation of nonlinear and dispersive waves, they do not possess traveling wave solutions that resemble undular bores. Apparently, dissipation is equally important to the dispersion and nonlinearity for the development of such traveling wave solutions.

In this paper we studied the conditions for the existence of dispersive and regularized shock waves for a dissipative Peregrine system, and we proved existence and uniqueness (up to horizontal translations) of such solutions. We showed that these waves can describe undular bores generated in laboratory experiments. We closed this paper by proving error estimates between the dissipative and the non-dissipative Peregrine systems. We conclude that the dispersive shock waves obtained by studying the non-dissipative Peregrine system can be accurate approximations of undular bores only for certain time scales.

\section*{Acknowledgments}

DM thanks KAUST for their hospitality during a visit when this work was initiated.

\appendix

\section{Derivation of dissipative Boussinesq equations}\label{appA}

In this appendix we outline a derivation of the dissipative Boussinesq system (\ref{eq:boussinesqs}) found in \cite{BCS2002,DD2007}. The starting point is the simplified Navier-Stokes equations of \cite{JW2004} that takes into account effects of  weak dissipation due to viscosity, without bottom friction and with potential flow, an idea also developed in \cite{WJ2006,DDZ2008}.  We model long water waves of characteristic wavelength $L$ and amplitude $A$ traveling over constant depth $D$,  in scaled space $x$, $z$ and time $t$ variables $x=x'/L$, $z=z'/D+1$, $t= \sqrt{gD}t'/L$,
where $x', z', t'$ denote the original (unscaled) variables. The scaled free-surface elevation above the undisturbed level of the water $\eta$ and velocity potential $\phi$ are formulated as 
$\eta(x,t)=\eta'(x',t')/A$, $\phi(x,z,t)=\sqrt{gD}\phi'(x',z',t')/gAL$.
In the long wave regime we assume that the ratios $\mu=D/L\ll 1$ and $\epsilon=A/D\ll 1$, while the Stokes number $S=\epsilon/\mu^2$ is $O(1)$.
Note that the bottom in the scaled vertical space variable $z$ is at $z=0$ and the free surface at $z=1+\epsilon\eta$.
For $x\in \mathbb{R}$, the scaled simplified Navier-Stokes equations of \cite{JW2004} for potential flow, along with the boundary conditions at the free surface and at the flat sea floor form the system of equations \cite{DDZ2008}
\begin{align}
&\phi_{xx}+\frac{1}{\mu^2}\phi_{zz}=0,\quad\text{ for } 0<z<1+\epsilon\eta\ ,\label{eq:A1}\\
&\eta_t+\epsilon \phi_x\eta_x-\frac{1}{\mu^2}\phi_z=0, \quad\text{ for }  z=1+\epsilon\eta\ ,\label{eq:A2}\\
&\phi_t+\eta+\frac{1}{2}\epsilon \phi_x^2+\frac{1}{2}\frac{\epsilon}{\mu^2}\phi_z^2+\frac{\epsilon}{\mu^2}\kappa\phi_{zz}=0,\quad\text{ for } z=1+\epsilon\eta \ ,\label{eq:A3}\\
&\phi_{zz}=0,\quad\text{ for } z=1+\epsilon\eta\ ,\label{eq:A4}\\
&\phi_z=0,\quad\text{ for }\quad z=0\ .\label{eq:A5} 
\end{align}
In the previous notation, $\kappa$ is proportional to the kinematic viscosity parameter of the Navier-Stokes equations.
Because of the incompressibility condition (\ref{eq:A1}) and the boundary condition (\ref{eq:A5}) the formal asymptotic expansion of the velocity potential takes the form \cite{BCS2002}
$$\begin{aligned}
\phi(x,z,t)&=\sum_{k=0}^\infty \frac{(-1)^k\mu^{2k}}{(2k)!}\frac{\partial^{2k}F(x,t)}{\partial x^{2k}}z^{2k}\\
&= F(x,t)-\frac{\mu^2}{2}z^2F_{xx}(x,t)+\frac{\mu^4}{24}z^4F_{xxxx}(x,t)+O(\mu^6)\ ,
\end{aligned}$$ where $F(x,t)=\phi(x,0,t)$ is the velocity potential evaluated at the bottom $z=0$. Substitution into the surface boundary condition (\ref{eq:A2}) and the Bernouli equation (\ref{eq:A3}) leads to the following system of equations for the scaled free-surface elevation $\eta$ and the horizontal velocity at the bottom $w=F_x$,
\begin{equation}\label{eq:Asys}
\begin{aligned}
&\eta_t+w_x+\epsilon (\eta w)_x-\frac{1}{6}\mu^2w_{xxx}=\text{terms quadratic in $\epsilon$, $\mu^2$}\\
&w_t+\eta_x+\epsilon ww_x-\frac{1}{2}\mu^2 w_{xxt}-\epsilon\kappa w_{xx}=\text{terms quadratic in $\epsilon$, $\mu^2$}
\end{aligned}
\end{equation}
We then evaluate the horizontal velocity $u(x,z,t)=\phi_x(x,z,t)$ at depth $z=\sqrt{3}/3$ to obtain
$u(x,t)=u(x,\sqrt{3}/3,t)=w(x,t)-\frac{1}{6}\mu^2w_{xx}(x,t)+O(\mu^4)$.
Solving the last relationship for $w$ taking into account that $u=w+O(\mu^2)$ we obtain 
\begin{equation}\label{eq:hv1}
w=u+\frac{1}{6}\mu^2u_{xx}+O(\mu^4)\ .
\end{equation}
Substituting (\ref{eq:hv1}) into the system (\ref{eq:Asys}), we obtain the equations
\begin{equation}\label{eq:Asys2}
\begin{aligned}
&\eta_t+u_x+\epsilon (\eta u)_x=\text{terms quadratic in $\epsilon$, $\mu^2$}\\
&u_t+\eta_x+\epsilon uu_x-\frac{1}{3}\mu^2 u_{xxt}-\epsilon\kappa u_{xx}=\text{terms quadratic in $\epsilon$, $\mu^2$}
\end{aligned}
\end{equation} 
Dropping the high order terms, changing the variables into the original, unscaled ones, and setting the parameter of the dissipative term equal to $\varepsilon$ we obtain the dissipative Boussinesq system (\ref{eq:boussinesqs}).

\bibliographystyle{plain} 

\end{document}